\documentclass[12pt]{article}

\usepackage{amssymb}
\usepackage{amstext}
\usepackage{amscd}
\usepackage[francais]{babel}

\newtheorem{teo}{Th\'eor\`eme}[section]
\newtheorem{prop}[teo]{Proposition}
\newtheorem{lem}[teo]{Lemme}
\newtheorem{cor}[teo]{Corollaire}
\newtheorem{conj}[teo]{Conjecture}

\newtheorem{defini}[teo]{D\'efinition}

\newtheorem{rem}[teo]{Remarque}

\newcommand{\GL}{{\rm GL}}

\newcommand{\Sh}{{\rm Sh}}
\newcommand{\GSp}{{\rm GSp}}

\newcommand{\Gal}{{\rm Gal}}
\newcommand{\Res}{{\rm Res}}

\newcommand{\Hom}{{\rm Hom}}

\newcommand{\Aut}{{\rm Aut}}

\newcommand{\Discr}{{\rm Discr}}

\newcommand{\FF}{{\mathbb F}}
\newcommand{\CC}{{\mathbb C}}
\newcommand{\RR}{{\mathbb R}}
\newcommand{\ZZ}{{\mathbb Z}}
\newcommand{\QQ}{{\mathbb Q}}
\newcommand{\NN}{{\mathbb N}}

\newcommand{\GG}{{\mathbb G}}
\newcommand{\SSS}{{\mathbb S}}
\newcommand{\AAA}{{\mathbb A}}

\newcommand{\lto}{\longrightarrow}

\def\Fp{\mathfrak{p}}

\def\Fg{\mathfrak{g}}
\def\Fh{\mathfrak{h}}

\newcommand{\cB}{{\cal B}}

\newcommand{\uT}{\underline{T}}

\newcommand{\ol}{\overline}

\newcommand{\oQ}{\overline{\QQ}}

\newcommand{\End}{{\rm End}}

\newenvironment{prf}[1]{\trivlist
\item[\hskip \labelsep{\it
#1.\hspace*{.3em}}]}{~\hspace{\fill}~$\square$\endtrivlist}
\newenvironment{proof}{\begin{prf}{\bf Preuve}}{\end{prf}}

\title{Nombre de classes des tores de multiplication complexe
 et bornes inf\'erieures pour les  orbites Galoisiennes de points sp\'eciaux.
\footnote{Ullmo : Universit\'e de Paris-Sud, Bat 425, 91405, Orsay 
Cedex France, e-mail : ullmo@math.u-psud.fr ; Yafaev : University 
College London, 
Department of Mathematics, 25 Gordon street, WC1H OAH, London, United 
Kingdom, e-mail : yafaev@math.ucl.ac.uk
2000 Mathematics Subject Classification : 11G18}
\footnote{To appear in Bulletin de la SMF}}
\author{ Emmanuel Ullmo \and Andrei Yafaev}

\date{}
\begin{document}
\maketitle

\tableofcontents
\bigskip

\bigskip

\section{Introduction}

Ce papier est motiv\'e par la conjecture d'Andr\'e-Oort dont voici l'\'enonc\'e.

\begin{conj}[Andr\'e-Oort]
Soit $S$ une variet\'e de Shimura et soit $\Sigma\subset S$ 
un ensemble de points  sp\'eciaux. Alors les composantes irr\'eductibles de l'adh\'erence
de Zariski de $\Sigma$ sont des sous-vari\'et\'es sp\'eciales de $S$.
\end{conj} 

Cette conjecture a \'et\'e r\'ecemment d\'emontr\'ee par Klingler et les deux auteurs 
\cite{UY}, \cite{KY}
en admettant l'hypoth\`ese de Riemann g\'en\'eralis\'ee. La strategie consistait 
\`a combiner des m\'ethodes  galoisiennes et g\'eom\'etriques d'Edixhoven
avec des techniques ergodiques de Clozel-Ullmo.
Tr\`es r\'ecemment, Jonathan Pila a mis en place une strat\'egie faisant
intervenir des id\'ees issues de la logique  pour attaquer la conjecture
d'Andr\'e-Oort \cite{Pi1}, \cite{PW}. Cette nouvelle approche
 a d\'ej\`a permi  \cite{Pi2} de 
d\'emontrer la conjecture d'Andr\'e-Oort pour des produits de courbes modulaires de mani\`ere inconditionnelle.
Qu'on adopte la strat\'egie d'Edixhoven ou celle de Pila, un des ingr\'edients majeurs
 est une borne inf\'erieure
suffisamment forte pour la taille des orbites sous Galois des points sp\'eciaux des vari\'et\'es de Shimura.

Il est \`a noter que la strat\'egie de Pila n\'ecessite des meilleures bornes que celles
requises par la m\'ethode  d'Edixhoven. La minoration de la taille des orbites
sous Galois des points sp\'eciaux  obtenue dans \cite{UY} d\'epend de la validit\'e
de l'hypoth\`ese de Riemann g\'en\'eralis\'ee et est insuffisante pour les applications \`a la m\'ethode de Pila. Notons que l'on ne sait pas \`a ce jour
 que sur l'espace de module $\AAA_{g}$ des vari\'et\'es ab\'eliennes principalement
polaris\'ees de dimension $g\ge 4$ il n'y a qu'un nombre fini de points correspondants
\`a des vari\'et\'es ab\'eliennes \`a multiplication complexe d\'efinies sur des extensions
de $\QQ$ de degr\'e born\'e.

Le but principal de cet article est d'obtenir des minorations pour la taille des orbites sous Galois
de point sp\'eciaux utilisables dans la strat\'egie de Pila. Nous obtenons en toute g\'en\'eralit\'e
ces bornes sous l'hypoth\`ese de Riemann g\'en\'eralis\'ee et dans certains cas  de mani\`ere inconditionnelle. Notons aussi qu'il n'est pas \'evident de pr\'evoir exactement
le type de bornes n\'ecessaire pour la m\'ethode de Pila mais nous
pensons que celles que nous obtenons sont difficilement am\'eliorables qualitativement
et probablement  adapt\'ees aux applications en vue.

Un point sp\'ecial $x$ d'une vari\'et\'e de Shimura d\'efinit un tore alg\'ebrique
$T$  sur $\QQ$ et un corps de nombres $E=E(x,T)$, le corps reflex. 
Le corps $E$ est un corps CM.
On dispose alors d'un morphisme de tores alg\'ebriques sur $\QQ$
dit de r\'eciprocit\'e
$$
r=r_{x,T}:R_{E}:=\Res_{E/\QQ}\GG_{m,E}\rightarrow T.
$$

Soit $M$ un tore alg\'ebrique sur $\QQ$. Soit $K_M^m$ le sous-groupe compact ouvert maximal de $M(\AAA_f)$.
Le groupe de classes $h_{M}$ de $M$ est par d\'efinition
le groupe fini
$$
h_M = M(\QQ)\backslash M(\AAA_f) / K^m_M.
$$ 

Le morphisme de r\'eciprocit\'e $r$ induit au niveau des groupes de classes
une application $\overline{r}: h_{R_{E}}\rightarrow h_{T}$ et l'orbite sous Galois
du point sp\'ecial $x$ est minor\'ee par le cardinal de  l'image de $\overline{r}$.
 Minorer la taille de l'orbite sous Galois de $x$ revient donc \`a minorer
le cardinal de l'image de $\overline{r}$.

La strat\'egie suivie dans ce papier est d'abord de minorer la
taille du groupe de classes de $T$ puis de borner la taille
du conoyau de $\overline{r}$. Les bornes obtenues 
pour la taille de $h_{T}$ sont inconditionnelles et ont la forme voulue.
Nous pouvons minorer  le conoyau de $\overline{r}$
quand   le noyau de $r$ est connexe de mani\`ere inconditionelle
et obtenir les bornes voulues pour l'orbite de Galois de $x$ dans ce cas. 
Nous donnons des crit\`eres assurant la connexit\'e du noyau de  $r$
dans la section \ref{sec4}. Par exemple ce noyau est toujours connexe
si $x$ est un point sp\'ecial de $\AAA_{g}$ pour $g\leq 3$ ou si $x$ est ``Galois g\'en\'erique ''
pour $g$ arbitraire.

Quand le noyau de $r$
n'est pas connexe, l'estimation du conoyau de $\overline{r}$
semble \^etre un probl\`eme s\'erieux de g\'eom\'etrie alg\'ebrique
et de cohomologie galoisienne que nous n'avons pas su r\'esoudre
sans l'hypoth\`ese de Riemann g\'en\'eralis\'ee.

Pr\'ecisons un peu la nature des r\'esultats obtenus. 
Soit $M$ un tore alg\'ebrique sur $\QQ$ de dimension $d$.
Soit $L$ le corps de d\'ecomposition de $M$
et $D_L$ la valeur absolue de son discriminant.
Notre but est de donner une borne inf\'erieure pour $h_M$ en fonction de $D_L$.
Soit $X^*(M)$ le groupe de caract\`eres de $M$ et $\chi_M$ le caract\`ere de
la repr\'esentation d'Artin correspondante de $G=\Gal(L/\QQ)$.
On consid\`ere la fonction $L$ d'Artin associ\'ee que l'on d\'enote
$L(s,M)$ et $\rho_M$ son quasi-r\'esidu dont la d\'efinition est donn\'ee \`a la section \ref{s2.1.2}.

On  d\'efinit ensuite le \emph{quasi-discriminant} $D_M$ de $M$.
Il est d\'efini comme le rapport entre  deux mesures de Haar sur $M(\AAA)$.  La proposition \ref{p3.1} donne
une formule ferm\'ee qui relie $D_{M}$ au conducteur d'Artin $a(M)$ du module Galoisien
$X^*(M)\otimes \QQ$ et au cardinal du groupe des composantes du mod\`ele de N\'eron de 
type fini de $M$ sur $\ZZ$.
 Shyr \cite{Sh} montre la formule:
\begin{equation}\label{Shyr}
h_M R_M = w_M \tau_M \rho_M D_M^{1/2}
\end{equation}
o\`u $w_M$ est la taille du `groupe des unit\'es de $M$',  $R_M$ le r\'egulateur de $M$
et $\tau_M$ le nombre de Tamagawa.
Il est \`a noter que dans le cas du tore $M=\Res_{F/\QQ}\GG_{mF}$ o\`u $F$ est un corps de nombres,
on retrouve  la formule classique pour le nombre de classes de l'anneau des entiers de $F$ .

En explicitant et en \'evaluant les invariants arithm\'etiques de $M$
intervenant dans la formule de Shyr (\ref{Shyr}) on montre  que
\begin{equation}\label{eqintro}
h_M R_M \gg D_L^{\mu}
\end{equation}
o\`u les constantes ne d\'ependent que de $d$ et sont explicites en fonction de $d$.
La forme pr\'ecise du r\'esultat est donn\'ee dans le th\'eor\`eme \ref{teo2.3}.
Une fois la formule ferm\'ee pour $D_{M}$ obtenue les r\'esultats principaux
sont une minoration de la forme voulue pour le conducteur d'Artin $a(M)$ (proposition \ref{p3.2}) et une
estimation de type Brauer-Siegel pour le quasi-r\'esidu $\rho_{M}$ (proposition \ref{p2.1}).
Il est \`a noter que dans le cas o\`u $M$ est un tore de multiplication complexe $T$
 le r\'egulateur $R_T$ est trivial et nous obtenons une minoration de $h_{T}$.

On applique ensuite notre formule pour $h_T$ au probl\`eme de minoration des orbites
Galoisiennes des points sp\'eciaux dans les vari\'et\'es de Shimura.

Soit $(G,X)$ une donn\'ee de Shimura, $K$ un sous-groupe
compact ouvert de $G(\AAA_{f})$ et $\Sh_{K}(G,X):=G(\QQ)\backslash (X\times G(\AAA_{f})/K)$
la vari\'et\'e de Shimura associ\'ee.
On peut sans perte de g\'en\'eralit\'e supposer que 
$G$ est semisimple de type adjoint.

Soit $x=\ol{(h,1)}$ un point sp\'ecial.
Alors le groupe de Mumford-Tate $T$ de $h$ est un tore de multiplication 
complexe.
Soit
$K_T$ le sous-groupe compact ouvert $T(\AAA_f)\cap K$ de $T(\AAA_f)$.
On dispose donc d'un morphisme de r\'eciprocit\'e
$$
r\colon \Res_{E/\QQ}\GG_{m,E} \lto T
$$
o\`u comme pr\'ec\'edemment $E$ d\'esigne le corps reflex de $(T,\{x\})$.
L'orbite Galoisienne  $O(x)$ de $x=\ol{(h,1)}$ a pour taille  le cardinal de l'image
de $r((E\otimes \AAA_f)^*)$ dans $T(\QQ)\backslash T(\AAA_f) /K_T$.
On d\'emontre alors dans la section \ref{s5.1} que
$$
\vert O(x) \vert \gg B^{i(T)}|K_T^m/K_T| \vert Im(\overline{r})\vert
$$
o\`u $Im(\overline{r})$ d\'esigne l'image
de $r((E\otimes \AAA_f)^*)$ dans $T(\QQ)\backslash T(\AAA_f) / K_T^m$,
$B$ est une constante ne d\'ependant que de la vari\'et\'e de Shimura
$\Sh_{K}(G,X)$ et $i(T)$ est le cardinal  de l'ensemble des nombres premiers $p$
tels que la projection de $K_{T}$ dans $T(\QQ_{p})$ n'est pas \'egale 
\`a $K_{T,p}^{m}$. 

Quand le point sp\'ecial $x$ varie dans $\Sh_{K}(G,X)$ parmi
les points tels que $r_{x}$ est \`a noyau connexe, un r\'esultat
de Clozel et du premier auteur \cite{CU} assure que 
le conoyau de $\overline{r}$ est uniform\'ement born\'e.
On obtient dans ce cas en utilisant l'\'equation (\ref{eqintro}) une minoration satisfaisante
de $\vert O(x)\vert$ sous la forme
\begin{equation}\label{eqi1}
\vert O(x) \vert \gg B^{i(T)}|K_T^m/K_T \vert D_{L}^{\mu}
\end{equation}
pour un $\mu$ explicite. C'est par exemple
le cas pour pour un point  du module des vari\'et\'es ab\'eliennes principalement polaris\'ees
correspondant \`a une vari\'et\'e ab\'elienne simple de dimension $g$ pour $g\leq 3$ ou pour un point \`a multiplication complexe 
 ``g\'en\'eral du point de vue Galoisien''
pour $g$ arbitraire.
 Les r\'esultats principaux
que nous obtenons dans cette direction sont donn\'es dans
la section \ref{s5.2}.

Un argument simple montre que
 le nombre de composantes connexes du noyau de $r_{x}$
est uniform\'ement born\'e quand $x$ varie parmi les points sp\'eciaux de
$\Sh_{K}(G,X)$. On en d\'eduit que l'image de $\overline{r}$ contient
l'image de l'\'elevation \`a la puissance $n$ de $h_{T}$ dans $h_{T}$
pour $n$ uniform\'ement born\'e.
 
En utilisant  l'hypoth\`ese de Riemann g\'en\'eralis\'ee on montre dans la section \ref{s6}  une minoration
de l'orbite sous Galois de $x$ de la forme voulue
\begin{equation}\label{eqi2}
\vert O(x) \vert \gg B^{i(T)}|K_T^m/K_T|  D_{L}^{\mu}.
\end{equation}
Ce r\'esultat est ind\'ependant des parties pr\'ec\'edentes.

Finalement, il est \`a noter que Tsimerman (voir \cite{Tsi}) a obtenu des r\'esultats comparables aux notres simultan\'ement.

\section{Pr\'eliminaires}
\subsection{Formule de classes g\'en\'eralis\'ee.}\label{s2.1}

Nous rappelons dans cette partie une formule due \`a Ono \cite{On1} et \cite{On2}
et Shyr \cite{Sh} donnant le nombre de classes
d'un tore alg\'ebrique $T$ sur $\QQ$ qui g\'en\'eralise la formule classique de 
Dedekind pour le nombre de classes de l'anneau des entiers d'un corps de nombres.
On d\'efinit et on estime le quasi-r\'esidu $\rho_{T}$ de $T$
qui intervient dans cette formule. On \'enonce un des r\'esultats 
principaux que nous avons en vue qui donne une minoration du produit
 du nombre de classes de $h_{T}$ de $T$ par le r\'egulateur $R_{T}$
 (th\'eor\`eme \ref{teo2.3}). Dans les applications \`a la multiplication
 complexe que nous avons en vue le r\'egulateur $R_{T}$ sera toujours \'egal \`a
 $1$ de sorte que l'on aura dans ce cas une minoration du nombre
 de classes $h_{T}$.

\subsubsection{Nombre de classes des tores.}\label{s2.1.1}

On note $\AAA$ (resp.  $\AAA_{f}$) l'anneau  des ad\`eles  (resp. des ad\`eles finis) de $\QQ$. 
Soit $G$ un groupe alg\'ebrique sur $\QQ$. Soit $K$ un sous-groupe compact
ouvert de $G(\AAA_{f})$, le nombre de classe $h_{G}(K)$ de $G$ relativement 
\`a $K$ est d\'efini comme le cardinal
de l'ensemble fini $G(\QQ)\backslash G(\AAA_{f})/K$ (\cite{PR} thm. 5.1). 

Si $T$ est un tore sur $\QQ$, 
et $p$ est premier  on note $K_{T,p}^m$  
 l'unique  sous-groupe compact ouvert maximal de $T(\QQ_{p})$.
Alors  $K_{T}^m:=\prod_{p} K_{T,p}^m$  est l'unique  sous-groupe compact ouvert maximal de $T(\AAA_{f})$.
Le nombre de classe $h_{T}$ de $T$ est d\'efini comme
\begin{equation}
h_{T}:=h_{T}(K_{T}^m).
\end{equation}

On note $K_{T,\infty}^m$ le sous-groupe compact maximal de $T(\RR)$. Le groupe
$T(\QQ)\cap K_{T,\infty}^m K_{T}^m$ est alors fini et on note
\begin{equation}
w_{T}=\vert T(\QQ)\cap K_{T,\infty}^m K_{T}^m\vert.
\end{equation}

\subsubsection{Fonction $L$ d'Artin de $T$ et estimations de $\rho_{T}$.}\label{s2.1.2}

Soit $T$ un tore alg\'ebrique sur $\QQ$, on note $X^*(T)$ le $\ZZ$-module libre
$\Hom(T_{\oQ},\GG_{m,\oQ})$ des caract\`eres de $T$. Pour toute extension $E$ de $\QQ$
on note $X^*(T)_{E}$ le sous-module de $X^*(T)$ form\'e des caract\`eres qui sont rationnels
sur $E$. Soit $L$ un corps de d\'ecomposition de $T$ et $G=\Gal(L/\QQ)$ le groupe de Galois de
$L$ sur $\QQ$. Le groupe $G$ agit sur $X^*(T)$ et on note $\chi_{T}$
le caract\`ere de cette repr\'esentation.

On note $L(s,T)=L(s,\chi_{T})$ la fonction $L$ d'Artin d\'efini par le $G$-module $X^*(T)\otimes \CC$.
On dispose d'un produit Eul\'erien $L(s,T)=\prod_{p}L_{p}(s,T)$, avec pour tout nombre
premier $p$
$$
L_{p}(s,T)=\det (1-p^{-s}Frob_{\Fp}\vert X^*(T)^{I_{\Fp}})^{-1}.
$$

Dans cette somme portant sur les nombres premiers $p$, $\Fp$
d\'esigne une place arbitraire de $L$ au dessus de $p$, $I_{\Fp}\subset G_{\Fp}$
d\'esigne le sous-groupe d'inertie  du groupe de d\'ecomposition $G_{\Fp}=\Gal(L_{\Fp}/\QQ_{p})$
et $Frob_{\Fp}\in G_{\Fp}/I_{\Fp}$ est le Frobenius en $\Fp$.

Soit $h$ le nombre de classes de conjugaison de $G$ et $\chi_{1},\dots,\chi_{h}$
les caract\`eres des repr\'esentations irr\'eductibles de $G$. On suppose
que $\chi_{1}$ est le caract\`ere de la repr\'esentation triviale. Par la th\'eorie
des fonctions $L$ d'Artin \cite{Art1} \cite{Art2}, si on a la d\'ecomposition
$\chi_{T}=\sum_{i=1}^hm_{i}\chi_{i}$ alors
$$
L(s,T)=\zeta(s)^{m_{1}}\prod_{j=2}^h L(s,\chi_{j})^{m_{j}}
$$
avec $\zeta(s)$ la fonction zeta de Riemann et $L(s,\chi_{i})$
la fonction $L$ d'Artin de $\chi_{i}$. Dans cette situation
\begin{equation}\label{rhoT}
\rho_{T}=\lim_{s\rightarrow 1} (s-1)^{m_{1}}L(s,T)=\prod_{i=2}^h L(1,\chi_{j})^{m_{j}}
\end{equation}
est fini et non nul.

Soit $H$ un sous-groupe de $G$  et $\chi_{H}$ le caract\`ere d'une repr\'esentation
de $H$. On note $\chi_{H}^{*}$ le caract\`ere de $G$ induit de $\chi_{H}$.
Par d\'efinition, on a donc
$$
\chi_{H}^{*}(\alpha)=\frac{1}{\vert H\vert}\sum_{g\in G}\chi'_{H}(g\alpha g^{-1})
$$ 
$\chi'_{H}$ d\'esignant l'extension de $\chi_{H}$ \`a $G$ nulle en dehors de $H$.

Soit $H_{1},\dots, H_{r}$, un syst\`eme de repr\'esentants des sous-groupes 
cycliques de $G$ \`a conjugaison pr\`es. On note ${\bf 1}_{H_{i}}$ le caract\`ere 
de la repr\'esentation triviale  de $H_{i}$.
Par la th\'eorie d'Artin,  (\cite{On1} 1.5.3) il existe des entiers naturels
$(m,\lambda_{i},\nu_{i})$ premiers entre eux dans leur ensemble qui sont d\'etermin\'es par
le $G$-module $X^*(T)$ tels que 
\begin{equation}
m\chi_{T}+\sum_{i=1}^r \lambda_{i} {\bf 1}_{H_{i}}^{*}=\sum_{i=1}^r \nu_{i} {\bf 1}_{H_{i}}^{*}.
\end{equation}

Soit $F_{i}$ les sous-corps de $L$ correspondants \`a $H_{i}$ par la th\'eorie de
Galois on en d\'eduit une isog\'enie (\cite{On1} thm. 1.5.1)
\begin{equation}\label{eq2.1.2}
T^{m}\times \prod_{i=1}^{r} (\Res_{F_{i}/\QQ}\GG_{m})^{\lambda_{i}}\simeq
\prod_{i=1}^r (\Res_{F_{i}/\QQ}\GG_{m})^{\nu_{i}}.
\end{equation}

On on aura besoin de l'\'enonc\'e suivant de type Brauer-Siegel concernant
la taille de $\rho_{T}$.

\begin{prop}\label{p2.1}
Soit $d$ un entier et $\epsilon$ un r\'eel positif.
Soit $T$ un tore alg\'ebrique sur $\QQ$ de dimension $d$.
Soit $L$ le corps de d\'ecomposition de $T$ et $D_{L}$ la valeur absolue
du discriminant de $L$. Il existe des constantes $c_{1}=c_{1}(d,\epsilon)$ et $c_{2}=c_{2}(d,\epsilon)$ (d\'ependants
uniquement de $d$ et $\epsilon$) telles que 
\begin{equation}
c_{1}D_{L}^{-\epsilon}   \leq \rho_{T}\leq c_{2}D_{L}^{\epsilon}.
\end{equation}
\end{prop} 
{\it Preuve.}
On d\'eduit de l'\'equation (\ref{eq2.1.2}) l' \'egalit\'e de fonctions $L$
$$
L(s,T)^m\prod_{i=1}^r \zeta_{F_{i}}(s)^{\lambda_{i}}=
\prod_{i=1}^r \zeta_{F_{i}}(s)^{\nu_{i}}
$$
o\`u l'on  note  $\zeta_E (s)$ la fonction z\^{e}ta d'un corps de nombres $E$.
Pour tout $i\in \{1,\dots,r\}$, $\zeta_{F_{i}}(s)$ a un p\^{o}le simple
de  r\'esidu not\'e $\rho_{F_{i}}$.   Soit $m_{1}$ l'ordre du p\^{o}le en $s=1$ de 
$L(s,T)$. On trouve que
$$
mm_{1}=\sum_{i=1}^r (\nu_{i}-\lambda_{i})
$$
et 
$$
\rho_{T}^m=\prod_{i=1}^{r} \rho_{F_{i}}^{\nu_{i}-\lambda_{i}}.
$$

Quand $T$ varie parmi les tores de dimension $d$, il n'y a qu'un nombre fini
de possibilit\'es pour le groupe de Galois $G$ comme groupe abstrait.
Quand $G$ est fix\'e, il n'y a qu'un nombre fini de possibilit\'es \`a isomorphismes
pr\`es  pour $X^*(T)\otimes \QQ$ comme $G$-module. Comme les entiers
$m, \lambda_{i},\nu_{i}$ ne d\'ependent que du $G$-module $X^*(T)\otimes \QQ$,
ils sont born\'es quand $T$ parcourt l'ensemble des tores alg\'ebriques
sur $\QQ$ de dimension $d$.

On est donc ramen\'e au lemme  suivant.

\begin{lem} Soit $d$ un entier et $\epsilon$ un r\'eel positif.
 Soit $L$ une extension
galoisiennne de $\QQ$ de degr\'e $d$ et soit $E$ une sous-extension de $L$.
Il existe des constantes $c'_{1}=c'_{1}(d,\epsilon)$ et  $c'_{2}=c'_{2}(d,\epsilon)$
(ne d\'ependants que de $d$ et $\epsilon$) telles que
$$
c'_{1}D_{L}^{-\epsilon}   \leq \rho_{E}\leq c'_{2}D_{L}^{\epsilon}.
$$
\end{lem}
D'apr\`es  (\cite{La} XVI-1) lemme 1, on sait que 
 $\rho_{E}\ll D_{E}^{\epsilon}$. On a par ailleurs la relation
$D_{L}=D_{E}^{[L:E]} N_{E/\QQ} (D_{L/E})$ o\`u
$D_{L/E}$ d\'esigne le discriminant relatif et $N_{E/\QQ}$
la norme de $E$ \`a $\QQ$. Ceci d\'emontre la majoration du 
lemme.

Pour la minoration, on utilise (\cite{La} XVI-2) th\'eor\`eme 2
qui assure que
 $$
 \rho_{E}\gg \rho_{L} D_{L}^{-\frac{\epsilon}{2}}
 $$
et (\cite{La} XVI-2) th\'eor\`eme 1 qui assure que
$$
\rho_{L}\gg D_{L}^{-\frac{\epsilon}{2}}.
$$

\subsubsection{Mesures de Haar sur $T(\AAA)$ et le quasi-discriminant $D_{T}$.}\label{s2.1.3}

Soit $T$ un tore alg\'ebrique  sur  $\QQ$ de rang $r$. Soit $v$ une place de $\QQ$
et $r_{v}$ le rang de $X^*(T)_{\QQ_{v}}$. Soit 
$\chi_{v,1},\dots, \chi_{v,r_{v}}$ une $\ZZ$-base de $X^*(T)_{\QQ_{v}}$ et 
$$
\pi_{v}: T(\QQ_{v})\rightarrow (\RR_{+}^{\times})^{r_{v}}
$$
$$
x\mapsto \pi_{v}(x)= (\vert \chi_{v,i}(x)\vert_{v})_{1\leq i\leq r_{v}}.
$$

Pour $v=\infty$, $\pi_{\infty}$ induit un isomorphisme
 $\overline{\pi_{\infty}}:T(\RR)/K_{T,\infty}^m\simeq ( \RR_{+}^{\times})^{r_{\infty}}$.
 On note 
 $$
 dt_{\infty}:=\overline{\pi_{\infty}}^* (\wedge_{i=1}^{r_{\infty}} \frac{dt_{i}}{t_{i}})
 $$
 et $\nu_{\infty}$ la mesure de Haar sur $T(\RR)$ amalgamant $dt_{\infty}$
 et la mesure de Haar normalis\'ee sur $K_{T,\infty}^m$.
 
 Pour $v=p$ fini,  $\pi_{p}$ induit un isomorphisme
 $\overline{\pi_{p}}:T(\QQ_{p})/K_{T,p}^m\simeq \ZZ^{r_{p}}$.
 On note $dt_{p}$ le pull-back  par $\overline{\pi_{p}}$ de la mesure
 discr\`ete sur  $\ZZ^{r_{p}}$ 
 et $\nu_{p}$ la mesure de Haar sur $T(\QQ_{p})$ amalgamant $dt_{p}$
 et la mesure de Haar normalis\'ee sur $K_{T,p}^m$.
On obtient ainsi une mesure de Haar
\begin{equation}
\nu_{T}:=\prod_{v\in \Sigma_{\QQ}} \nu_{v}
\end{equation}
sur $T(\AAA)$.

Soit $\omega$ une forme diff\'erentielle $\QQ$-rationelle non nulle de degr\'e maximal sur $T$.
Pour tout $v\in \Sigma_{\QQ}$, $\omega$ induit une mesure de Haar $\vert \omega_{v}\vert$
sur $T(\QQ_{v})$. On sait alors que 
\begin{equation}
\vert \omega_{T}\vert :=\vert \omega_{\infty}\vert 
\prod_{p\in \Sigma_{\QQ,f}} L_{p}(1,T) \vert \omega_{p}\vert
\end{equation} 
d\'efinit une mesure de Haar dite de Tamagawa sur  $T(\AAA)$ ind\'ependante du choix de $\omega$.

Il existe alors une constante positive $c_{T}$ telle que $\vert \omega_{T}\vert=c_{T}\nu_{T}$.
On appelle  alors le quasi-discriminant de $T$ le nombre  
\begin{equation}
D_{T}:=\frac{1}{c_{T}^2}.
\end{equation}
 
\subsubsection{Formule de classes d'un tore alg\'ebrique.}\label{s2.1.4}
 
 Soit $T$ un tore sur $\QQ$,  Shyr \cite{Sh} montre la formule de classes
 suivante:
 \begin{equation}\label{eqcl}
 h_{T}R_{T}=w_{T}\tau_{T} \rho_{T} D_{T}^{\frac{1}{2}}.
 \end{equation}
 
 Dans cette formule $R_{T}$ est le r\'egulateur de $T$ d\'efini comme 
 le covolume de l'image du  r\'eseau des unit\'es $T(\QQ)\cap K_{T}^m$
 dans $\RR^{r_{\infty}-r}$ (\cite{On1}, p. 131) et 
 $\tau_{T}$ est le nombre de Tamagawa de $T$ ( \cite{On1}, 3.5).
 
 Quand $T=\Res_{F/\QQ}\GG_{m,F}$ pour un corps de nombres $F$,
 on v\'erifie que $h_{T}=h_{F}$ est le nombre de classes de $F$,
 $R_{T}=R_{F}$ est le r\'egulateur de $F$, $w_{T}=w_{F}$ le nombre de
 racines de l'unit\'e de $F$, $\rho_{T}=\rho_{F}$ est le r\'esidu en $1$
 de la fonction z\^{e}ta du corps $F$. Soit $r$ (resp. $s$) le nombre de places r\'eelles
 (resp. complexes \`a conjugaison complexe pr\`es) de $F$ et $D_{F}$ la valeur absolue
 du  discriminant de $F$.
 Alors $D_{T}=\frac{D_{F}}{2^{2r}(2\pi)^{2s}}$. Comme dans cette situation
 $\tau_{T}=1$, on retrouve la formule classique de la th\'eorie des nombres:
 $$
  h_{F}R_{F}=2^{-r}(2\pi)^{-s}w_{F} \rho_{F} D_{F}^{\frac{1}{2}}.
 $$
 
\subsubsection{Minoration du nombre de classes d'un tore alg\'ebrique.}\label{s2.1.5}
  
  Nous pouvons maintenant \'enoncer un des r\'esultats principaux que nous avons
  en vue.
  
  \begin{teo}\label{teo2.3}
 Soit $d$ un entier positif. Il existe des constantes positives
 $\lambda(d)$ et $B(d)$ telles que pour tout tore alg\'ebrique $T$
 sur $\QQ$ de dimension $d$ et tout $\epsilon>0$, il existe une constante
 positive $c(d,\epsilon)$ telle que
 \begin{equation}
 h_{T}R_{T}\ge c(d,\epsilon) B(d)^{i(L)}D_{L}^{\frac{\lambda(d)}{2}-\epsilon}.
 \end{equation} 
 Dans cette \'equation $L$ d\'esigne le corps de d\'ecomposition de $T$,
 $D_{L}$ d\'esigne la valeur absolue du discriminant de $L$ et 
 $i(L)$ d\'esigne le nombre de nombres premiers divisant $D_{L}$.
  \end{teo}
  
  Si le r\'egulateur $R_{T}$ est trivial on obtient une minoration de $h_{T}$.
  On verra que c'est le cas pour les tores associ\'es \`a la multiplication complexe.
  La constante $\lambda(d)$ est explicit\'ee
  dans la d\'efinition \ref{d3.4}. Elle est facilement 
calculable pour $d$ fix\'e et on pourrait l'\'etudier quand $d$ varie.
Un calcul \`a la main donne par exemple
$\lambda(1)=1$, $\lambda(2)=\lambda(3)=\frac{2}{5}$, $\lambda(4)=\lambda(5)=\frac{4}{11}$
et $\lambda(6)=\lambda(7)=\frac{1}{5}$. Un test de parit\'e simple dans la d\'efinition de $\lambda(n)$
montre que pour tout $n\in \NN$,  $\lambda(2n)=\lambda(2n+1)$.

  La preuve du th\'eor\`eme
  sera donn\'ee dans la section \ref{s3.4} et sera une cons\'equence
  simple d'estimations sur les invariants $w_{T}, \tau_{T},\rho_{T}$ et $D_{T}$
  intervenant dans la formule de Shyr (\ref{eqcl}).
  
\subsection{Mod\`eles entiers des tores}\label{s2.2}
 
\subsubsection{Conducteurs d'Artin}
 
 Soit $p$ un nombre premier et
 $K$ une extension finie de $\QQ_{p}$. On note $O_{K}$ son anneau d'entiers,
 $m_{K}$ l'id\'eal maximal de $O_{K}$ et $\kappa_{K}=\FF_{q}$ son corps r\'esiduel. On note
 $O_{K}^{hs}$ le hens\'elis\'e stricte de  $O_{K}$ et $K^{hs}$ le corps des fractions
 de  $O_{K}^{hs}$. Soit $\pi$ une uniformisante de $O_{K}$, on normalise la valuation
 $v$ de $K$ de sorte que $v(\pi)=1$. La valeur absolue associ\'ee est alors
 $\vert \alpha\vert=q^{-v(\alpha)}$.
 
 Soit $L$ une extension finie galoisienne de $K$.
 On note $G=\Gal(L/K)$ le groupe de Galois de $L$ sur $K$.
 Soit 
 $$
\Delta_{-1}:= G\supset \Delta_{0}\supset \Delta_{1}\supset \dots 
 $$
 la filtration d\'ecroissante de ramification avec $\Delta_{0}=I$ le sous-groupe d'inertie
 et $\Delta_{1}$ le sous-groupe d'inertie sauvage. On note $g_{i}=\vert \Delta_{i}\vert$.
 
 Soit $V$ une repr\'esentation lin\'eaire complexe de dimension finie
 de $G$, le conducteur d'Artin $a(V)$ de $V$ est d\'efini (\cite{Se} VI-2)
 par
 \begin{equation}
 a(V):=\sum_{i\ge 0} \frac{g_{i}}{g_{0}} \dim(V/V^{\Delta_{i}}).
 \end{equation}
 Quand $V$ est mod\'er\'ement ramifi\'e 
 $a(V)=\dim(V/V^{I})$ et $a(V)=0$ quand $V$ est non ramifi\'e.  
 
\subsubsection{Mod\`eles entiers des tores}\label{s2.2.1}
 
 On fixe encore  un nombre premier $p$ et  une extension finie  $K$ de $\QQ_{p}$. 
 Soit $T$ un tore sur $K$ de dimension $d$. Il existe un mod\`ele $\uT$ de $T$ sur $O_{K}$
  de type fini  et lisse sur $O_{K}$ tel que $\uT(O_{K}^{hs})$ est le sous-groupe
  maximal born\'e de $T(K^{hs})$. On dira que $\uT$ est le mod\`ele
  de N\'eron-Raynaud de type fini de $T$. Dans cette situation $\uT(O_{K})$
  est le sous-groupe compact maximal de $T(K)$.
  
  Le mod\`ele de N\'eron-Raynaud $\uT^{NR}$ d\'efini dans \cite{BLR} est un mod\`ele
  lisse de $T$ sur $O_{K}$ tel que $\uT^{NR}(O_{K}^{hs})=T(K^{hs})$. Ce mod\`ele
  est localement de type fini sur $O_{K}$.

  Les composantes  de l'\'el\'ement neutre de $\uT$ et $\uT^{NR}$ 
  co\"\i ncident. On note $\uT^0$ la composante de l'\'el\'ement neutre
  de $\uT$ et $\phi(\uT):=\uT/\uT^0$ le groupe des composantes connexes
  de $\uT$. Alors $\phi(\uT)$ est un sch\'ema en groupe fini \'etale et est d\'etermin\'e
  par le $G$-module $\phi(\uT)(\oQ_{p})$.   
  
  Soit $L$ le corps de d\'ecomposition de $T$. Soit 
  $$
  R:=\Res_{L/K} T_{L}=\Res_{L/K}\GG_{m,L}^d
  $$
  alors
  $$
 \underline{R}=\Res_{O_{L}/O_{K}} \GG_{m,O_{L}}^d.
 $$
 
 Le tore $T$ se plonge canoniquement dans $R$.

Soit $T_{O_{K}}$ la fermeture sch\'ematique de $T$
dans $\underline{R}$, alors $\uT$ s'obtient \`a partir de 
$T_{O_{K}}$ par un proc\'ed\'e de lissage d\'ecrit dans
\cite{BLR}, voir aussi (\cite{CY} section 3). Si l'extension
$L$ de $K$ est mod\'er\'ement ramifi\'ee par le th\'eor\`eme 4.2
de \cite{Ed}, $T_{O_{K}}$ est lisse et  co\"\i ncide
donc avec $\uT$.
  
  On aura besoin du r\'esultat suivant
  dans la suite.

  \begin{prop}\label{prop2.4}
  Il existe une constante $\Phi=\Phi(d)$ (ne d\'ependant que de $d$)
  telle que pour tout tore $T$ de dimension
  au plus $d$ sur $K$
  \begin{equation}
\vert \phi(\uT)\vert \leq \Phi(d).
  \end{equation}
  \end{prop}

Comme le mod\`ele de N\'eron-Raynaud (de type fini ou non) commute au changement de base
non ramifi\'e, on peut pour calculer  $\vert \phi(\uT)\vert $ supposer que $K=K^{hs}$. Si $L$ est le corps de d\'ecomposition 
de $T$, alors $G=\Gal(L/K)=I$ est le groupe d'inertie.

Soit $T_{s}$ le sous-tore d\'eploy\'e maximal de $T$
et $T_{a}$ le tore  quotient anisotrope $T/T_{s}$. D'apr\`es
\cite{CY} lemme 11.2, en passant aux mod\`eles de N\'eron-Raynaud
de type fini sur $O_{K}$, on obtient une suite exacte courte
$$
1\rightarrow\underline{T_{s}}\rightarrow \uT\rightarrow \underline{T_{a}}\rightarrow 0.
$$

\begin{lem}
On a l'\'egalit\'e $\phi(\uT)=\phi(\uT_{a})$.
\end{lem}

{\it Preuve}. Comme $\uT_{s}=\GG_{m,O_{K}}^{d'}$ pour un certain $d'\leq d$,
$\uT_{s}$ est connexe et la suite exacte pr\'ec\'edente
induit une suite exacte courte 
$$
1\rightarrow\underline{T_{s}}^0\rightarrow \uT^0\rightarrow \underline{T_{a}}^0\rightarrow 0.
$$
Le lemme s'obtient alors par une application du lemme du serpent.

On peut donc supposer que $T$ est anisotrope sur $K$.
Alors  $X^*(T)^{I}=0$ et par \cite{Xa}  cor. 2.19 on en d\'eduit que $\vert \phi(\uT)\vert= \vert H^{1}(I,X^*(T))\vert$. Comme $T$ est de dimension $d_{1}\leq d$, et $L$ est le corps de d\'ecomposition
de $K$, $I=\Gal(L/K)$ agit fid\`element sur $X^*(T)=\ZZ^{d_{1}}$. 
Le nombre de sous-groupes finis de $\GL_{d_{1},\ZZ}$ \`a conjugaison pr\`es
est fini. Il existe donc qu'un nombre fini de possibilit\'es pour  $\vert H^{1}(I,X^*(T))\vert$.
Ceci termine la preuve de la proposition \ref{prop2.4}.

\subsubsection{Mesure de Haar et conducteur d'Artin}\label{s2.2.2}
 
 On garde les notations de la section pr\'ec\'edente.
 On note $\omega(\uT)$ la puissance ext\'erieure maximale de
 l'espace des $1$-formes diff\'erentielles sur $\uT$. C'est 
 un $O_{K}$-module libre de rang $1$ et on choisit un
 g\'en\'erateur ${\omega}$ de $\omega(\uT)$. 
La mesure de Haar  $\vert {\omega}\vert $ 
 associ\'ee sur $T(K)$ est ind\'ependante du choix de $\omega$.  
 
 Fixons un isomorphisme $\phi: T_{L}\rightarrow \GG_{m,L}^d$.
 Le mod\`ele de N\'eron-Raynaud de type fini
 de $\GG_{m,L}^d$ est $\GG_{m,O_{L}}^d$.
 Soit ${\omega_{0}}$ un g\'en\'erateur de $\phi^*(\omega(\GG_{m,O_{L}}^d))$
 et $\vert \omega_{0}\vert$ la mesure de Haar sur $T(L)$ associ\'ee.
 
 Soit $a(T)$  le conducteur d'Artin de la repr\'esentation
 $X^*(T)\otimes \QQ$ de $\Gal (L/K)$.
 Gross et Gan (\cite{GG} sections 4-5 )  montrent qu'il
 existe une classe $\Theta\in O_{K}/O_{K}^{\times 2}$  telle que 
 $$
 v(\Theta)=a(T)
 $$
 et telle que $\frac{{\omega_{0}}}{\sqrt{\Theta}}$
 soit  rationnelle sur $K$.  Par le th\'eor\`eme 7.3 de $\cite{GG}$
 on a 
 \begin{equation}
 \vert {\omega}\vert=\frac{{\vert \omega_{0}}\vert}{\vert \sqrt{\Theta}\vert}
 \end{equation}  
 
 La proposition 4.7 de \cite{Gr} montre alors que
 \begin{equation}\label{eq2.2.2}
 \int_{\uT^0(O_{F})} L_{p}(1,T)\vert \underline{\omega}\vert=1.
 \end{equation}
 Le r\'esultat de Gross est en fait beaucoup plus g\'en\'eral.
 Il s'applique \`a un groupe r\'eductif arbitraire sur $K$ et aux mod\`eles
 entiers sur $O_{K}$ donn\'es par la th\'eorie de Bruhat-Tits. 
 Dans le cas d'un tore, ces mod\`eles co\"\i ncident
 avec les mod\`eles de N\'eron-Raynaud de type fini.
 Dans les notations de  \cite{Gr} le motif $M$ qui appara\^\i t
 est juste $X^*(T)$ pour un tore de sorte que l'on a la relation
 $\det(1-F\vert M^{\vee}(1)^{I})=\det(1-\frac{F}{p}\vert X^*(T)^{I})=L_{p}(1,T)^{-1}$.
 
\section{Le quasi-discriminant $D_{T}$.}\label{s.3}

 Soit $T$ un tore sur $\QQ$ de dimension $d$.
 Soit $L$ le corps de d\'ecomposition de $T$. 
 Le but de cette partie est de donner une formule ferm\'ee
 pour le quasi-discriminant $D_{T}$ de $T$ et de comparer
 $D_{T}$ \`a la valeur absolue du discriminant $D_{L}$ de $L$.
 
 Soit $\uT$ le mod\`ele de N\'eron-Raynaud de type fini de $T$ sur $\ZZ$.
 Pour tout nombre premier $p$,
 $\uT_{p}:=\uT_{\ZZ_{p}}$ est le mod\`ele de N\'eron-Raynaud de type fini 
 de $T_{\QQ_{p}}$ sur $\ZZ_{p}$  d\'ecrit dans la section \ref{s2.2.1}.
 Soit $\phi_{T,p}:=\phi(\uT_{p})$ le groupe des composantes de $\uT_{p}$.
 On note $a(T)$ le conducteur d'Artin de $X^*(T)\otimes \QQ$.
 Donc $a(T)=\prod_{p}p^{a_{p}(T)}$
 o\`u $a_{p}(T)$ est le conducteur d'Artin de $X^*(T_{\QQ_{p}})\otimes \QQ$.
 
 Soit $\SSS:=\Res_{\CC/\RR}\GG_{m,\CC}$ le tore de Deligne.
 On a une d\'ecomposition en produit direct sous la forme
\begin{equation}\label{eq3.2}
 T_{\RR}=\GG_{m,\RR}^{a}\times \SSS^b\times SO(2)_{\RR}^c
\end{equation}
 avec $\dim(T)=a+2b+c$ (\cite{Vo}, p. 106).
 
\subsection{Une formule ferm\'ee pour $D_{T}$}\label{s.3.1}
 Le but de cette section est de montrer la formule ferm\'ee suivante pour
 le quasi-discriminant $D_{T}$ de $T$.
 
 \begin{prop}\label{p3.1}
 On a 
 \begin{equation}
 D_{T}= \frac{a(T)}{2^{2a}(2\pi)^{2b+2c}\prod_{p}\vert \phi_{T,p}(\FF_{p})\vert^2}
 \end{equation}
 \end{prop}

{\it Preuve.} 
 Soit $\omega(\uT)$ la puissance ext\'erieure maximale de l'espace des
 $1$-formes diff\'erentielles invariantes sur $\uT$. Alors $\omega(\uT)$
 est un $\ZZ$-module libre de rang $1$ dont on fixe un g\'en\'erateur $\omega$.
 
 Soit $L$ le corps de d\'ecomposition de $T$. Fixons un isomorphisme
 $\phi:T_{L}\rightarrow \GG_{m,L}^d$. Soit $\alpha_{0}$ un g\'en\'erateur du 
 $\ZZ$-module libre de rang $1$ $\omega(\GG_{m,\ZZ}^d)$ et
 $\omega_{0}:=\phi^*\alpha_{0}$. D'apr\`es \cite{GG} cor. 3.7,
 il existe $D\in \QQ^*/\QQ^{*2}$ tel que $\frac{\omega_{0}}{\sqrt{D}}$
 est rationnel sur $\QQ$.  
 
 On peut fixer un repr\'esentant $D$ dans $\QQ^*$ de la mani\`ere suivante.
 Pour tout nombre premier $p$, on note $\omega_{p}$ et $\omega_{0,p}$ les 
 diff\'erentielles invariantes sur $T_{\QQ_{p}}$ induites par $\omega$ et $\omega_{0}$.
 D'apr\`es les r\'esultats de la section \ref{s2.2.2}, il existe 
 $\Theta_{p}\in \QQ_{p}/\QQ_{p}^{*2}$,
 v\'erifant $v_{p}(\Theta_{p})=a(X^*(T_{\QQ_{p}}))$,
  tel que 
  \begin{equation}\label{eq3.1}
  \frac{\omega_{0,p}}{\sqrt{\Theta_{p}}}
  \end{equation}
 soit rationnel sur $\QQ_{p}$ et tel que 
 $$
 \vert \omega_{p}\vert=\frac{\vert \omega_{0,p}\vert}{\vert \sqrt{\Theta_{p}}\vert}.
 $$
 
 Comme $\frac{\omega_{0,p}}{\sqrt D}$ et $\frac{\omega_{0,p}}{\sqrt{\Theta_{p}}}$
 sont  deux formes diff\'erentielles invariantes rationnelles sur $T_{\QQ_{p}}$ on voit que 
 $v_{p}(D)-v_{p}(\Theta_{p})$ est pair. En changeant $D$ en $Dp^{2k}$
 pour un $k$ convenable, on peut supposer que 
 $$
 v_{p}(D)=v_{p}(\Theta_{p})=a(X^*(T_{\QQ_{p}})).
 $$
 On peut donc en notant $a_{p}=a(X^*(T_{\QQ_{p}}))$ 
 choisir $D$ au signe pr\`es sous la forme
 $$
 D=\epsilon(T)\prod_{p} p^{a_{p}}=\epsilon(T) a(T)
 $$
 avec $\epsilon(T)=\pm 1$.
 
 On suppose dans la suite que $D$ est ainsi normalis\'e. 
 Le signe $\epsilon(T)$ d\'epend de $T_{\RR}$. 
 Un calcul direct utilisant la d\'ecomposition 
 (\ref{eq3.2}) de $T_{\RR}$ montre que $i^{b+c} \omega_{0,\infty}$ est 
 rationnelle sur $T_{\RR}$. Comme la forme diff\'erentielle $\frac{\omega_{0,\infty}}{\sqrt{\epsilon(T)}}$
 est aussi rationnelle sur $T_{\RR}$, on en d\'eduit que $\epsilon(T)=(-1)^{b+c}$.

Un choix normalis\'e de $D$ est donc 
$$
D=(-1)^{b+c} a(T).
$$

Soit $\nu_{T}=\nu_{\infty}\prod_{p}\nu_{p}$ la mesure de Haar sur
$T(\AAA_{f})$ d\'efinie dans la section \ref{s2.1.3}. 
En utilisant la forme invariante $\QQ$-rationelle
 $\frac{\omega_{0}}{\sqrt{D}}$ dans la d\'efinition de $\vert \omega_{T}\vert$
 et l'\'equation (\ref{eq3.1}), on trouve
$$
\vert \omega_{T}\vert =\frac{\vert \omega_{0,\infty}\vert}{\sqrt{a(T)}}\prod_{p}
 \vert \omega_{p}\vert L_{p}(1,T)
$$
et par d\'efinition
$$
\nu_{T}=\sqrt{D_{T}} \vert \omega_{T}\vert.
$$

Un calcul simple aux places r\'eelles utilisant la d\'ecomposition (\ref{eq3.2}) de $T_{\RR}$
montre que 
$$
\vert \omega_{0,\infty}\vert=2^{a}(2\pi)^{c+b}\nu_{T,\infty}.
$$

Fixons un compact $U_{\infty}$ de $T(\RR)$ d'int\'erieur non vide et suffisamment r\'egulier.
Alors
$$
\int_{U_{\infty}\times \prod_{p}K_{T,p}^m}\vert \omega_{T}\vert=
\frac{1}{2^{a} (2\pi)^{b+c} \sqrt{D_{T}}} \int_{U_{\infty}} \vert \omega_{0,\infty}\vert
=\int_{U_{\infty}} \frac{\vert \omega_{0,\infty}\vert}{\sqrt{a(T)}}\prod_{p}\int_{K_{T,p}^m}
\vert\omega_{T,p}\vert L_{p}(T,1).
$$

Par d\'efinition du mod\`ele de N\'eron-Raynaud de type fini de $T$,
$\uT_{p} (\ZZ_{p})=K_{T,p}^m$
pour tout premier $p$.

On en d\'eduit que  
$$
\int_{K_{T,p}^m}
\vert\omega_{T,p}\vert L_{p}(T,1)=\vert \phi_{T,p}(\FF_{p})\vert\int_{\uT_{p}^0(\ZZ_{p})}
\vert\omega_{T,p}\vert L_{p}(T,1)=\vert \phi_{T,p}(\FF_{p})\vert
$$
en utilisant l'\'equation (\ref{eq2.2.2}). Ceci termine la preuve
de la proposition \ref{p3.1}

\subsection{Minoration de $a(T)$ et de $D_{T}$.}\label{s.3.2}

On conserve les notations de la section \ref{s.3}.
Le but de cette partie est de montrer les minorations suivantes
du conducteur d'Artin $a(T)$  et du quasi-discriminant $D_{T}$ de  $T$.

\begin{prop}\label{p3.2}
Soit $d$ un entier positif. Il existe de constantes $c(d)$, $c'(d)$, $A(d)$ et $\lambda(d)$
strictement positives
telles que pour tout tore $T$ sur $\QQ$ de dimension $d$ dont le corps de
d\'ecomposition est $L$, on a 
\begin{equation}
a(T)\ge c(d) D_{L}^{\lambda(d)}.
\end{equation}
et 
\begin{equation}
D_{T}\ge c'(d) A(d)^{i(L)}D_{L}^{\lambda(d)}.
 \end{equation} 
 Dans cette derni\`ere \'equation 
 $i(L)$ d\'esigne le nombre de nombres premiers divisant $D_{L}$.
\end{prop}

La minoration de $D_{T}$ est une cons\'equence de la minoration
de $a(T)$, de la formule ferm\'ee pour $D_{T}$ donn\'ee \`a la proposition
\ref{p3.1} et de la majoration du cardinal du groupe des composantes
du mod\`ele de N\'eron $\uT$ de $T$ obtenue \`a la proposition 2.4.
On peut prendre $A(d)=\frac{1}{\Phi(d)^2}$ (avec les notations de 2.4)
et $c'(d)=\frac{c(d)}{(2\pi)^{2d}}$.

Les constantes $\lambda(d)$ et $c(d)$ sont explicit\'ees 
dans les d\'efinitions \ref{d3.4} et \ref{d3.6}.

On utilisera essentiellement le lemme \'el\'ementaire suivant.

\begin{lem}\label{l3.3}

Soit $X$ un $\QQ$-vectoriel de dimension $d$ muni d'une action fid\`ele
$\rho$ d'un groupe cyclique $I=\ZZ/n\ZZ$ pour un entier $n=\prod p^{n_{p}}$.
Soit $\phi(x)$ la fonction indicatrice d'Euler.
Le nombre $n_{\rho}$ 
de caract\`ere non triviaux de $I$ intervenant dans $X\otimes \CC$
v\'erifie
$$
n_{\rho}\ge \sum_{p\vert n}\phi(p^{n_{p}})-\epsilon(n)
$$
avec $\epsilon(n)=0$ si $n_{2}\neq 1$ ou $n=2$ et $\epsilon(n)=1$
sinon.
En particulier $ \sum_{p\vert n} \phi(p^{n_{p}})-\epsilon(n)\leq d$.
\end{lem}

{\it Preuve.} Soit $\sigma$ un g\'en\'erateur de $I$.
Alors $\rho(\sigma)$ est diagonalisable dans $X\otimes \CC$.
Les valeurs propres de $\sigma$ sont des racines $n$-i\`emes
de l'unit\'e. Soit $d_{1}, d_{2},\dots,d_{r}$ les ordres des valeurs propres
de $\sigma$. Comme l'action est fid\`ele les $d_{i}$  ne sont  pas \'egaux \`a $1$.
Si $(x_{1}, \dots, x_{s})$ est un $s$-tuple   d'entiers on note
$p(x_{1},\dots,x_{s})$ 
le plus petit commun multiple des $x_{i}$.
Comme $\sigma$ est d'ordre $n$, on a $p(d_{1},\dots,d_{r})=n$.
Si $\sigma$ admet une valeur propre d'ordre $d_{i}$, l'irr\'eductibilit\'e
des polyn\^{o}mes cyclotomiques sur $\QQ$ montre que 
$\sigma$ admet toutes les racines primitives $d_{i}$-i\`eme de l'unit\'e
comme valeurs propres.  Comme les valeurs propres de $\sigma$
d\'eterminent des caract\`eres de $I$ intervenant dans $X\otimes \CC$,
on en d\'eduit que 
$$
n_{\rho}\ge \sum_{i=1}^r \phi(d_{i}).
$$   
Il suffit donc de minorer le second membre de cette in\'egalit\'e par
$$
\sum_{p\vert n}\phi(p^{n_{p}})-\epsilon(n)
$$
quand $(d_{1},\dots,d_{r})$ varie parmi les diviseurs de $n$ tels
que  $p(d_{1},\dots,d_{r})=n$. 

On note $v_{p}$ la valuation $p$-adique normalis\'ee sur 
 $\QQ$. Pour tout nombre premier $p$ divisant $n$, il existe un entier 
 $i\in \{1,\dots,r\}$ tel que $v_{p}(d_{i})=v_{p}(n)$. En divisant les entiers
 $d_{i}$ convenablement, on peut supposer que pour tout $p$ divisant $n$ et pour tout
 $i\in \{1,\dots,r\}$,
  $v_{p}(d_{i})=v_{p}(n)$ ou  $v_{p}(d_{i})=0$ et qu'il existe un indice $i$
  tel que $v_{p}(d_{i})=v_{p}(n)$.
 On obtient alors le r\'esultat en remarquant que si $a$ et $b$
 sont des entiers premiers entre eux $\phi(ab)=\phi(a)\phi(b)\ge \phi(a)+\phi(b)$
 d\`es que $\min(a,b)\ge 3$ et que si $a$ est impair
 $\phi(2a)=\phi(a)=\phi(a)+\phi(2)-1$.

\begin{defini}\label{d3.4}
On d\'efinit les fonctions sur $\NN^*$,
$$
\psi(s):= \max \{n=\prod p^{n_{p}}\in \NN\vert \ \sum_{p\vert n}\phi(p^{n_{p}})-\epsilon(n)\leq s\}
$$
et
\begin{equation}
\lambda(s):= \min_{n=\prod p^{n_{p}}\leq \psi(s)} \frac{\sum_{p\vert n}\phi(p^{n_{p}})-\epsilon(n)}{n-1}.
\end{equation}
\end{defini}

\begin{lem}\label{l.3.5}
Soit $T$ un tore sur $\QQ$ de dimension $d$, $L$ son corps de d\'ecomposition
et $R_{L}:=\Res_{L/\QQ} \GG_{m,L}$.
Soit $\Fp$ une place de
$L$ divisant un premier $p$. On suppose  l'extension locale $L_{\Fp}/\QQ_{p}$
mod\'er\'ement ramifi\'ee.  On note $D_{\Fp}$ et $I_{\Fp}$ le groupe de d\'ecomposition
et le groupe d'inertie en $\Fp$. Soit $a_{p}(T)$ et $a_{p}(R_{L})$ les conducteurs d'Artin
des  $D_{\Fp}$-modules $X^*(T_{\QQ_{p}})$ et  $X^*(R_{L,\QQ_{p}})$ respectivement. Alors
\begin{equation}
a_{p}(T)\ge \lambda(d) a_{p}(R_{L}).
\end{equation}
\end{lem}

{\it Preuve.} Soit $\chi_{T}$ le caract\`ere du  $I_{\Fp}$-modules $X^*(T_{\QQ_{p}})\otimes \CC$
et $\chi_{R}$ celui de $X^*(R_{L,\QQ_{p}})\otimes \CC$. Soit $a_{I_{\Fp}}$ le caract\`ere
de la repr\'esentation d'Artin de $I_{\Fp}$.  Si on note $<\ ,\ >_{I_{\Fp}}$
le produit scalaire hermitien canonique sur l'espace des fonctions centrales
sur $I_{\Fp}$, alors d'apr\`es Serre (\cite{Se}, VI-2),
$$
a_{p}(T)=<\chi_{T} ,a_{I_{\Fp}} >_{I_{\Fp}},\ \ \mbox{ et } 
a_{p}(R_{L})=<\chi_{R} ,a_{I_{\Fp}} >_{I_{\Fp}}.
$$

Comme l'extension $L_{\Fp}/\QQ_{p}$ est mod\'er\'ement ramifi\'ee
le groupe d'inertie $I_{\Fp}$ est cyclique d'ordre $e_{\Fp}$.
Comme $L$ est le corps de d\'ecomposition de 
$T$, $Gal(L/\QQ)$ agit fid\`element sur $X^*(T)$.
Comme $I_{\Fp}$ est un sous-groupe de $Gal(L/\QQ)$,
il agit aussi fid\`element sur $X^*(T)$. Dans cette situation
$a_{I_{\Fp}}$ est le caract\`ere de la repr\'esentation d'augmentation
de $I_{\Fp}$ (\cite{Se}, prop. 2, p. 108). Par d\'efinition $a_{I_{\Fp}}$
est donc la somme des caract\`eres non triviaux de $I_{\Fp}=\ZZ/e_{\Fp}\ZZ$.

On en d\'eduit que 
$a_{p}(T)$ est minor\'e par le nombre de caract\`eres distincts
non triviaux de $I_{\Fp}$ apparaissant dans $X^*(T)\otimes \CC$.
Le r\'esultat est alors une application du lemme \ref{l3.3}
et du fait que $a_{p}(R_{L})=e_{\Fp}-1$.

\begin{defini}\label{d3.6}
Soit $s$ un entier, on note $\alpha(s)$ l'ordre maximal d'un sous-groupe fini de
$\GL_{s}(\QQ)$. 
%On d\'efinit une entier $\beta(s)$ comme le maximum 
%des $v_{p}(D_{L})$ quand $L$ parcours l'ensemble des extensions de $\QQ$
%de degr\'e au plus $\alpha(s)$ et $p$ l'ensemble des nombres premiers. 
%On a not\'e $D_{L}$ le discriminant de $L$ et $v_{p}$ la valuation $p$-adique
%usuelle. La finitude de $\beta(s)$ r\'esulte par exemple de (\cite{Se} Rem. p .67).

On d\'efinit alors 
\begin{equation}
c(s)=\prod_{p\leq s+1}\frac{1}{p^{\alpha(s)^{2}}}.
\end{equation}
Des bornes tr\`es pr\'ecises sur $\alpha(s)$ sont donn\'ees dans
\cite{Fr} et \cite{Fe}.
\end{defini}

On peut maintenant d\'emontrer la proposition \ref{p3.2}
avec les constantes $\lambda(d)$ et $c(d)$ de \ref{d3.4} et \ref{d3.6}.
D'apr\`es \cite{Se} VI-3, on a
$$
a(T)=\prod_{p}p^{a_{p}(T)} \mbox{ et } a(R_{L})=\prod_{p}p^{a_{p}(R_{L})}.
$$
Si $p>d+1$, soit $\Fp$ une place de $L$ au dessus de $p$.
Le groupe d'inertie $I_{\Fp}$ agit fid\`element dans $X^*(T)$
car $L$ est le corps de d\'ecomposition de $T$. Le lemme \ref{l3.3}
assure qu'il 
n'y a pas  de sous-groupes cycliques d'ordre $p$ dans $I_{\Fp}$.
La ramification en $\Fp$ est donc mod\'er\'ee et par le lemme \ref{l.3.5}
$$
p^{a_{p}(T)}\ge (p^{a_{p}(R_{L})})^{\lambda(d)}
$$

Soit $p$ un nombre premier plus petit que $d+1$.
Comme $L$ est le corps de d\'ecomposition de $T$, 
le groupe de Galois $G$ agit fid\`element dans $X^{*}(T)$.
En particulier $\dim(X^{*}(R_{L}))=[L:\QQ]=\vert G\vert\leq \alpha(d)$
et une majoration simple du conducteur d'Artin utilisant la d\'efinition 
donne $a_{p}(R_{L})\leq \alpha(d)^{2}$.

Comme $a_{p}(T)\ge 0$ et $\lambda(d)\leq 1$ on obtient
$$
p^{a_{p}(T)}\ge \frac{1}{p^{\alpha(d)^{2}}} (p^{a_{p}(R_{L})})^{\lambda(d)}.
$$
Comme $D_{L}=a(R_{L})$, on finit la preuve de la proposition
\ref{p3.2} en combinant les r\'esultats obtenus pour les diff\'erents
nombres premiers $p$.

\subsection{Les invariants  cohomologiques de $T$.}

On garde les notations pr\'ec\'edentes, en particulier $T$ est un tore sur $\QQ$
de dimension fix\'ee $r$. Le but de cette partie est de donner des bornes uniformes
pour la taille du  nombre de Tamagawa $\tau_{T}$ de $T$ . Les bornes que nous avons en vue seront 
 des cons\'equences imm\'ediates de l'interpr\'etation cohomologique de cette 
 quantit\'e.
 
 On utilise les notations de l'appendice de Kottwitz et Shelstad \cite{KS} 
 concernant la dualit\'e de Tate-Nakayama.  On note donc pour tout tore
 $S$ sur $\QQ$
 
 $$
 H^{i}(\AAA,S):=H^{i}(\QQ,S(\overline{\AAA}))
 $$
 $$
  H^{i}(\AAA/\QQ,S):=H^{i}(\QQ,S(\overline{\AAA})/S(\oQ))
 $$
 et $ker^{i}(\QQ,S)$ le noyau de l'application naturelle $H^{i}(\QQ,S)\rightarrow 
 H^{i}(\AAA,S)$. On note $\tilde{H}^{i}(.,.)$ les groupe de cohomologie modifi\'es 
 \`a la Tate correspondants.
 
 \subsubsection{Estimation de $\tau_{T}$.}
 
 Le but de cette partie est de montrer l'estimation de $\tau_{T}$ suivante.
 
 \begin{prop}\label{p3.7}
 Il existe des constantes positives $c_{1}(r)$, et $c_{2}(r)$ ne d\'ependant que de $r$
 telles que pour tout tore $T$ sur $\QQ$ de dimension $r$:
 \begin{equation}
 c_{1}(r)\leq \tau_{T}\leq c_{2}(r).
 \end{equation}
 \end{prop}
 
 D'apr\`es le r\'esultat principal de Ono \cite{On2} on a 
 \begin{equation}\label{eqOno}
 \tau_{T}=\frac{\vert H^{1}(\QQ,X^*(T))\vert }{\vert ker^{1}(\QQ,T)\vert}.
 \end{equation}
 
 On a $H^1(\QQ,X^*(T))=H^1(\QQ,\ZZ^r)$ et $\Gal(\oQ/\QQ)$
 op\`ere via un sous-groupe fini  $G$ de $\GL_{r}(\ZZ)$.
 En utilisant la suite d'inflation-restriction (\cite{Se} VII-6), on v\'erifie que
 $ H^1(\QQ,\ZZ^r)=H^1(G,\ZZ^r)$. A $r$ fix\'e, il n'y a, \`a isomophisme pr\`es,
 qu'un nombre fini de choix pour $G$. On en d\'eduit 
 une borne uniforme pour $\vert H^1(\QQ, X^*(T))\vert $ en fonction de $r$.
 
 Par ailleurs, un calcul de cohomologie Galoisienne utilisant la dualit\'e de Nakayama-Tate
 (\cite{On2} 2.2--2.3) montre que $ker^1(\QQ,T)$ est un quotient de 
 $H^2(\QQ,X^*(T))$. Comme pr\'ec\'edemment on v\'erifie que
 $H^2(\QQ,X^*(T))=H^2(G,\ZZ^r)$ pour un sous-goupe $G$ de $\GL_{n}(\ZZ)$.
 On en d\'eduit que  $\vert ker^1(\QQ,T)\vert$ est aussi uniform\'ement born\'e en fonction 
 de $r$. Ceci termine la preuve de la proposition au vu de l'expression
 de $\tau_{T}$ dans  l'\'equation (\ref{eqOno}).
 
 \subsection{Preuve du th\'eor\`eme \ref{teo2.3}}\label{s3.4}
 
 Il s'agit juste de collecter les r\'esultats des parties pr\'ec\'edentes
 en partant de  la formule de classe de Shyr pour $T$ 
 donn\'ee \`a l'\'equation (\ref{eqcl}): 
$$
 h_{T}R_{T}=w_{T}\tau_{T} \rho_{T} D_{T}^{\frac{1}{2}}.
$$
On remarque que $w(T)\ge 1$ car  c'est un entier.
On pourrait borner $w(T)$ en fonction de $d$ mais nous
n'en n'auront pas d'usages. Les minorations de 
$\rho(T)$ (prop. \ref{p2.1}), de $D_{T}$ (prop. \ref{p3.2}) et de
$\tau_{T}$ (prop. \ref{p3.7}) permettent de finir la
minoration de $h_{T}R_{T}$ recherch\'ee au th\'eor\`eme  \ref{teo2.3}.

\section{Connexit\'e du noyau des morphismes de r\'eciprocit\'es.}\label{sec4}

\subsection{Corps de multiplication complexe et types CM.}

Soit $E$ un corps CM de degr\'e $2g$ sur $\QQ$. Soit
 $F$ son sous-corps totalement r\'eel maximal.
On note $E^{c}$ (resp. $F^c$ ) la cl\^{o}ture Galoisienne de $E$ (resp. $F$) 
et $\rho\in \Aut(E^{c}/\QQ)$ la conjugaison complexe. On pose
$$
J=\Hom(E,\oQ)=\Gal(E^{c}/E)\backslash \Gal(E^{c}/\QQ)
$$
et on fixe un type CM $\Sigma\subset J$ de sorte que 
$$
J=\Sigma\cup \Sigma^{\rho} \mbox{ et } \Sigma\cap \Sigma^{\rho}=\emptyset.
$$

On notera contrairement aux parties pr\'ec\'edentes $\Fg:=\Gal(E^c/\QQ)$
r\'eservant la lettre $G$ pour les groupes r\'eductifs intervenants  dans la suite.
Le groupe de Galois $\Fg$ op\`ere \`a droite transitivement et 
fid\`element sur $J$. On indexe les \'el\'ements de $\Sigma$ par
les indices $\{1,\dots,g\}$ et ceux de $\Sigma^{\rho}$ par les indices
 $\{-1,\dots,-g\}$ avec la convention que $k.\rho=-k$ pour tout
 $k\in \{1,\dots,g\}$.  Soit $C_{g}:= (\ZZ/2\ZZ)^g \rtimes S_{g}$ 
 le centralisateur de $\rho$ dans $S_{J}=S_{2g}$. Par convention $C_{g}$ et $S_{g}$
 op\`erent \`a droite
 sur $J$ par permutation. Pour $\sigma\in S_{g}$ et tout $k\in \{1,\dots,g\}$ on a
  $k.\sigma=\sigma^{-1}(k)$
 et $(-k).\sigma=-\sigma^{-1}(k)$.
 
 On dispose du r\'esultat suivant de Dodson (\cite{Do} 1.1).

\begin{prop}\label{p4.1}
(a) On a une suite exacte 
$$
1\rightarrow (\ZZ/2\ZZ)^v\rightarrow \Fg\rightarrow \Fg_{0}\rightarrow 1.
$$
Dans cette suite $(\ZZ/2\ZZ)^v$ est identifi\'e au sous-groupe 
$\Gal(E^c/F^c)$ de $\Fg$ et $\Fg_{0}:=\Gal(F^c/\QQ)$. On a toujours
$v\ge 1$ car le groupe engendr\'e par $\rho$ est un sous-groupe
de $(\ZZ/2\ZZ)^v$.

(b) Le groupe $\Fg$ muni de son action sur $J$ s'identifie
\`a un sous-groupe de $C_{g}$. Le groupe $\Fg_{0}$ s'identifie
\`a un sous-groupe de $S_{g}$, il agit transitivement sur $\{1,\dots,g\}$
et il agit sur le sous-groupe $(\ZZ/2\ZZ)^v\subset  (\ZZ/2\ZZ)^g$
de $C_{g}$ par permutation des coordonn\'ees. 

(c) On peut  \'ecrire
\begin{equation}\label{imp}
\Fg=\cup_{\sigma\in \Fg_{0}}(\ZZ/2\ZZ)^v (s(\sigma), \sigma)\subset C_{g}
\end{equation}
o\`u $(s(\sigma),\sigma)$ est un rel\'evement arbitraire de $\sigma$.
Soit $j:  (\ZZ/2\ZZ)^g \rightarrow  (\ZZ/2\ZZ)^g/(\ZZ/2\ZZ)^v$
la surjection canonique. Alors $js:\Fg_{0}\rightarrow (\ZZ/2\ZZ)^g/(\ZZ/2\ZZ)^v$
est un $1$-cocycle.
\end{prop} 

 On note $\Fh=\Gal(E^c/E)$ et on pose
 $$
 \tilde{\Sigma}:=\{\alpha\in \Fg,\Fh \alpha\in \Sigma\}=\sqcup_{i=1}^g \Fh \alpha_{i}.
 $$

 On rappelle que le type CM $\Sigma$ est dit primitif si il ne provient pas
 d'un type CM sur un sous-corps CM stricte de $E$. Si $\Sigma$ est primitif
 alors
 $$
 \Fh=\{\alpha\in \Fg,\alpha\tilde{\Sigma}=\tilde{\Sigma}\}.
 $$
 
Soit
 $$
\Fh':=\{\alpha\in \Fg, \tilde{\Sigma}\alpha=\tilde{\Sigma} \}:=
\{\alpha\in \Fg,\alpha\tilde{\Sigma'}=\tilde{\Sigma'}\}
$$
avec $\tilde{\Sigma}':=\tilde{\Sigma}^{-1}=\{x^{-1},\ x\in \tilde{\Sigma}\}$.
Soit $E'$ le sous-corps de $E^c$ attach\'e \`a $\Fh'$ par la correspondance de
Galois. Alors $E'$ est le corps reflex de $(E,\Sigma)$, c'est un corps CM.
On pose $[E':\QQ]:=2g'$. Soit $\Sigma'$ l'image de 
$\tilde{\Sigma'}$ dans 
$$
\Hom(E',\oQ)=\Fh'\backslash \Fg.
$$
Alors $\Sigma'$ est un type CM primitif sur $E'$.
On dit que $(E',\Sigma')$ est le dual de $(E,\Sigma)$.
Si $(E,\Sigma)$ est primitif, alors $(E,\Sigma)$
est le dual de $(E',\Sigma')$.

Dans la description (\ref{imp}) de $\Fg$ comme sous-groupe de $C_{g}$,
le fixateur de $\Sigma$ est le sous-groupe des \'el\'ements de la forme
$(1,\sigma)$ de $\Fg$.  Soit 
$$
\Fg_{\Sigma}:=\{\sigma\in \Fg_{0}, \ s(\sigma)\in (\ZZ/2\ZZ)^v\}.
$$
Alors $\Fh'=\{(1,\sigma), \ \sigma\in \Fg_{\Sigma}\}$
En particulier on a
$$
[E':\QQ]=2g'=2^v [\Fg_{0}:\Fg_{\Sigma}].
$$

\begin{rem}\label{r4.2}
Si on change le type CM $\Sigma$ en $\Sigma. b$ pour $b\in (\ZZ/2\ZZ)^g$
cela revient \`a conjuguer $\Fg$ dans $S_{2g}$ par $b$. On change alors
la section $s(\sigma)$ par multiplication par le cobord $ b\sigma(b)(\ZZ/2\ZZ)^v$
dans la description de $\Fg$ comme sous-groupe de $C_{g}$ de la proposition
\ref{p4.1}--c. 
\end{rem}

On utilisera le r\'esultat suivant de Dodson (\cite{Do} 2.1.2).

\begin{prop}\label{dodson}
Soit $E$ un corps CM de degr\'e $2g$ sur $\QQ$. Le cocycle
$js:\Fg_{0}\rightarrow (\ZZ/2\ZZ)^g/(\ZZ/2\ZZ)^v$ de \ref{p4.1}--c est trivial si et seulement si
il existe un type CM, $\Sigma_{1}$ sur $E$ dont le corps reflex $E_{1}$ est de degr\'e $2^v$
sur $\QQ$.
\end{prop}

Remarquer que la nullit\'e du cocycle $js$ est ind\'ependante du choix d'un type
CM $\Sigma$ au vu de la remarque pr\'ec\'edente. Noter que dans ce cas 
l'image de $js$ dans $H^2(\Fg_{0},(\ZZ/2\ZZ)^v)=Ext^1(\Fg_{0},(\ZZ/2\ZZ)^v)$
est nul. La suite exacte  de la proposition \ref{p4.1}-a est alors scind\'ee.
 Nous utiliserons
la proposition pr\'ec\'edente dans le cas suivant

\begin{lem}\label{l4.4}
Soit $E$ un corps CM de degr\'e $2g$ contenant un corps quadratique imaginaire $E'$.
Alors $v=1$ et le cocycle $js$ est trivial.
\end{lem}

Dans cette situation $E=FE'$ et $E^c=F^cE'$ donc $[E^c:F^c]=2$ et $v=1$.
Soit $\Sigma'$ un type CM  sur $E'$ et $\Sigma$ 
son extension \`a  $E$. Avec les notations pr\'ec\'edentes $\tilde{\Sigma}=\Fh=\tilde{\Sigma}^{-1}$.
On en d\'eduit que  $E'$ est le corps reflex de $(E,\Sigma)$ et que
le cocycle $js$ est trivial, d'apr\`es la proposition \ref{dodson}.

\subsection{Tores de multiplication complexe, le cas du groupe symplectique.}

Fixons la forme bilin\'eaire altern\'ee $\psi_{g}$ de matrice 
$$
J_g=\left(\matrix{
0 &-1_g\cr
1_g &0\cr}\right)
$$
sur $\QQ^{2g}$ et notons $G=\GSp_{2g}$ le groupe de similitudes symplectiques
associ\'e. On garde les notations de la section pr\'ec\'edente concernant le corps de
multiplication complexe $E$. On peut trouver un \'el\'ement $\iota\in E$
tel que $\iota^{\rho}=-\iota$. Ainsi $E$ est muni de la forme $\QQ$-lin\'eaire
altern\'ee
$$
<x,y>=Tr_{E/\QQ}(x^{\rho}\iota y).
$$
On peut alors fixer un isomorphisme symplectique 
$(E,<\ ,\ >)\simeq (\QQ^{2g},\psi_{g})$.

Soit $R_{E}:=\Res_{E/\QQ} \GG_{m,E}$. L'espace des caract\`eres $X^*(R_{E})$
s'\'ecrit alors
$$
X^*(R_{E})=\oplus_{\phi\in \Hom(E,\oQ)}\ZZ \phi=\oplus_{\alpha\in \Fh\backslash \Fg }\ZZ \Fh\alpha.
$$
Il sera utile de l'\'ecrire de la mani\`ere suivante qui fait intervenir le type CM $\Sigma$:
$$
X^*(R_{E})=\oplus_{i=1}^g \ZZ [\alpha_{i}]
\oplus \oplus_{i=1}^g  \ZZ[\overline{\alpha_{i}}]=\oplus_{i=1}^g \ZZ [\alpha_{i}]
\oplus \oplus_{i=1}^g  \ZZ[\alpha_{-i}]
$$
o\`u l'on a not\'e $[\alpha_{i}]$ l'\'el\'ement $\Fh \alpha_{i}$
du type CM $\Sigma\subset \Hom(E,\oQ)$
et $[\alpha_{-i}]=[\overline{\alpha_{i}}]$ l'\'el\'ement $\Fh \alpha_{i}\rho$
de $\Sigma^{\rho}$.  On a une description identique pour le r\'eseau 
$X_{*}(R_{E})$ des cocaract\`eres de $R_{E}$.

Soit $U_{E}$ le sous-tore de $R_{E}$ d\'efini par
$$
U_{E}:=\{x\in R_{E}, xx^{\rho}=1\}. 
$$
Soit  $GU_{E}$ le sous-tore
de $R_{E}$ en gendr\'e par $U_{E}$ et $\GG_{m,\QQ}\subset R_{E}$. 
Le tore $GU_{E}$ s'identifie \`a un tore maximal de $G$.

La th\'eorie de Deligne construit un param\`etre de Hodge
$h:\SSS\rightarrow R_{E}\otimes \RR$ se factorisant par $GU_{E}\otimes \RR\subset G_{\RR}$.

Le module des cocaract\`eres $X_{*}(GU_{E})$ de $GU_{E}$ est le sous-module de $X_{*}(R_{E})$
qui se d\'ecrit par
$$
X_{*}(GU_{E})=\{\sum_{i=1}^g n_{i} [\alpha_{i}]+n_{-i}[\alpha_{-i}], \ n_{i}+n_{-i}=n_{j}+n_{-j}
\mbox{ pour tout $i$, $j$}\}.
$$

Le cocaract\`ere $\mu=\mu_{h}$ de $GU_{E}$ associ\'e au param\`etre de Hodge
est dans cette description $\mu=\sum_{i=1}^g[\alpha_{i}]$. Si on note
$e_{i}=[\alpha_{i}]-[\alpha_{-i}]$ on a 
$$
X_{*}(GU_{E})=\ZZ \mu\oplus \oplus_{i=1}^g\ZZ e_{i}.
$$

On rappelle que $E'$ d\'esigne le corps reflex de $(E,\Sigma)$
et que $\Fh'$ est le sous-groupe de $\Fg$ associ\'e \`a $E'$ par la correspondance de Galois. Le morphisme
de r\'eciprocit\'e 
\begin{equation}
r:R_{E'}\rightarrow R_{E}
\end{equation}
se factorise par $GU_{E}$.
Il se d\'ecrit au niveau des cocaract\`eres de la mani\`ere suivante.
$$
X_{*}(r): X_{*}(R_{E'})= \oplus_{\beta\in \Fh'\backslash \Fg }\Fh'\beta \longrightarrow X_{*}(R_{E})
$$ 
$$
\Fh' \beta\mapsto \sum_{i=1}^g \Fh \alpha_{i}\beta=\sum_{i=1}^g \Fh \alpha_{i.\beta}=\sum_{i=1}^g [\alpha_{i.\beta}]
$$
o\`u $\beta$ agit sur $\{\pm 1,\pm 2,\dots,\pm g\}$ via la description de $\Fg$ donn\'e \`a la proposition \ref{p4.1}.

On d\'efinit le sous-module $L_{\mu}:= X_{*}(r) (X_{*}(R'_{E}))$ de $X_{*}(GU_{E})$.
Soit $L'_{\mu}:=(L_{\mu}\otimes \QQ)\cap X_{*}(GU_{E})$. Alors $L'_{\mu}$
est un sous-$\ZZ$-module galoisien satur\'e de $X_{*}(GU_{E})$.
Le sous-tore  $M=MT(\mu)$ de $GU_{E}$ associ\'e \`a $L'_{\mu}$ est le groupe de
Mumford-Tate de $\mu$ (ou de $h$). Par d\'efinition $M$ est le plus petit $\QQ$-sous-tore
de $GU_{E}$ tel que $\mu_{\CC}$ se factorise par $M_{\CC}$. Le lemme suivant est alors
une cons\'equence de l'\'equivalence de cat\'egories entre la cat\'egorie des tores alg\'ebriques
et celle des $\ZZ$-modules  galoisiens libres de rang fini:

\begin{lem}
Le morphisme de r\'eciprocit\'e $r$ est  \`a noyau connexe si et seulement si 
$L_{\mu}=L'_{\mu}$.
\end{lem} 

Voici un crit\`ere simple qui assure la connexit\'e de $r$.

\begin{prop}\label{p4.3}
Si dans la suite exacte
$$
1\rightarrow (\ZZ/2\ZZ)^v\rightarrow \Fg\rightarrow \Fg_{0}\rightarrow 1
$$
de la proposition \ref{p4.1} on a $v=g$, alors $L_{\mu}=X_{*}(GU_{E})$
et $r$ est \`a noyau connexe. 
\end{prop}

{\it Preuve.} Dans cette situation le groupe $\Fg$ contient les transpositions
$\tau_{k}=(k,-k)$ pour $k\in \{1,\dots,g\}$. On en d\'eduit que $L_{\mu}$ contient
$\mu-\mu.\tau_{k}=e_{k}$. Donc $L_{\mu}=L'_{\mu}=X_{*}(GU_{E})$ et 
$r$ est \`a noyau connexe. 

Ce r\'esultat pr\'ecise  un \'enonc\'e  de Clozel et du  premier auteur ( \cite{CU}, sec. 3.2)
o\`u le cas o\`u $\Fg=C_{g}$ est obtenu. Noter que ce dernier cas, appel\'e ``Galois g\'en\'erique''
dans \cite{CU} est le cas g\'en\'erique comme expliqu\'e dans (\cite{CU}, sec. 2). 

\begin{prop}\label{p4.7}
Si $g\leq 3$ le morphisme de r\'eciprocit\'e $r$ est \`a noyau connexe. 
\end{prop}
{\it Preuve.} Le cas $g=1$ est bien connu. On a alors $E=E'$ et $r$ est un isomorphisme.

Dans le cas  $g=2$, si  $v=2$, on peut appliquer la proposition \ref{p4.3}. On peut donc supposer
que $v=1$. Le groupe $\Fg_{0}=S_{2}$ est engendr\'e par la transposition $\sigma=(1,2)$.
Le choix de $s$ n'\'etant bien d\'efini qu'\`a multiplication par $\rho$ pr\`es, on peut  supposer
que $s(\sigma)=Id$ o\`u que $s(\sigma)$ est la transposition $(1,-1)$. Dans le premier cas
$L_{\mu}=\ZZ\mu\oplus \ZZ (e_{1}+e_{2})=L'_{\mu}$ et dans le second 
$L_{\mu}=X_{*}(GU_{E})$.  Noter que si le type CM $\Sigma$ est primitif, le r\'esultat de
Ribet (\cite{Ri}, 3.5) nous assure que le rang de $L_{\mu}$ est $3$ donc que le premier cas
n'intervient pas.

Pour $g=3$ on peut comme pr\'ec\'edemment supposer que $1\leq v<3$. 
Le cas $v=2$ est en fait exclu. En effet $\Fg_{0}$ est soit $S_{3}$ soit le groupe
altern\'ee $A_{3}$. Dans tous les cas il contient le trois-cycle $\sigma_{0}=(1,2,3)$.
Les seuls points fixes de $\sigma_{0}$ dans son action sur $(\ZZ/2\ZZ)^3$ sont
$Id$ et $\rho$. Comme $\sigma_{0}$ pr\'eserve $(\ZZ/2\ZZ)^v$ on trouve
que $v=2$ est impossible.   Pour $g=p$ un nombre premier arbitraire
cet argument montre que $p$ divise $2^{v}-2$.

Quand $v=1$, $(\ZZ/2\ZZ)^v=\{Id,\rho\}$ est central dans $\Fg$. dans cette situation
la suite exacte
$$
1\rightarrow \{Id,\rho\}\rightarrow \Fg\rightarrow \Fg_{0}\rightarrow 1
$$
est scind\'ee car $\rho$ est de signature $-1$ donc $\Fg\cap A_{2g}$ fournit un
scindage. La sous-extension de $E^c$ associ\'ee \`a $\Fg\cap A_{2g}$
est  un corps  quadratique imaginaire $\QQ[\sqrt{-\delta}]$ et intervient comme un corps reflex de $E$
pour un    type CM  de $E$. Par le corollaire 2.1.2 de \cite{Do}
on en d\'eduit que le cocycle $js:\Fg_{0}\rightarrow (\ZZ/2\ZZ)^g/ \{Id,\rho\}$
de la proposition \ref{p4.1} est trivial. Il existe donc $b\in  (\ZZ/2\ZZ)^g$ tel que 
$js(\sigma)=b\sigma(b)\{Id,\rho\}$ pour tout $\sigma\in \Fg_{0}$.
 On peut alors choisir $s(\sigma)=b\sigma(b)$.

Si $b\in \{Id,\rho\}$, $s(\sigma)=Id$ pour tout $\sigma\in \Fg_{0}$
alors $L_{\mu}$ est engendr\'e par $\mu$ et $\mu^{\rho}$.
On en d\'eduit que $L_{\mu}=\ZZ\mu\oplus\ZZ(e_{1}+e_{2}+e_{3})=L'_{\mu}$.
Dans cette situation le corps reflex est $\QQ[\sqrt{-\delta}]$ et le type
CM $\Sigma$ n'est pas primitif par le r\'esultat de Ribet (\cite{Ri}, 3.5).

Si $b\notin\{Id,\rho\}$, $(s(\sigma),\sigma)$ agit sur $\mu$
via l'action de $b.\sigma^{-1}(b)$. On remarque $b\sigma_{0}^{-1}(b)$
est une permutation paire non triviale. C'est donc un produit de $2$
transpositions. Donc  $\rho b\sigma_{0}^{-1}(b)=(i,-i)$ pour un $i\in \{1,2,3\}$
et $\rho b\sigma_{0}^{-2}(b)= \rho b \sigma_{0}(b)=(j,-j)$ avec $j\neq i$.
On en d\'eduit que $L_{\mu}$ contient $e_{i}=\mu-\mu.(i,-i)$ et $e_{j}$.
Comme il contient $\mu-\mu. \rho=e_{1}+e_{2}+e_{3}$, on trouve
que $L_{\mu}=X_{*}(GU_{E})$.

\begin{rem}
Le lecteur int\'er\'ess\'e par la complexit\'e combinatoire peut regarder la table des cas
possibles donn\'e par Dodson (\cite{Do} p.23) dans le cas $g=4$. On peut construire des
exemples de corps CM de degr\'e $8$ et de type CM $\Sigma$ tel que $L'_{\mu}/L_{\mu}\simeq \ZZ/2\ZZ$.
Par exemple on suppose que  $\Fg_{0}\simeq (\ZZ/2\ZZ)^2$ est le groupe de Klein engendr\'e par les doubles transpositions et si on
a une suite exacte scind\'e
$$
1\rightarrow \{1,\rho\}\rightarrow \Fg\rightarrow \Fg_{0}\rightarrow 1
$$
tel que $s(\sigma)=b\sigma(b)$ avec $b$ la transposition $(1,-1)$. Un calcul simple montre
que 
$$
L_{\mu}=\{\sum_{i=1}^4n_{i}e_{i}+r\mu\vert \  \sum_{i=1}^4 n_{i}\in 2\ZZ\} 
$$
qui est d'indice $2$ dans $L'_{\mu}=X_{*}(GU_{E})$.
En particulier $r$ n'est pas \`a noyau connexe dans ce cas.
\end{rem}

La proposition suivante montre que l'indice de $L_{\mu}$ dans $L'_{\mu}$ peut \^{e}tre divisible par des entiers arbitrairement
grands:

\begin{prop}
Soit $p$ un nombre premier impair. Il existe un corps CM $E$ de degr\'e $2p$ et un type CM sur $E$
tel que $L'_{\mu}/L_{\mu}\simeq \ZZ/(p-2)\ZZ$. 
\end{prop}
{\it Preuve.}
Soit $F$ un corps totalement r\'eel qui est une extension galoisienne
de $\QQ$ de groupe $ \ZZ/p\ZZ$. De tels $F$ existent comme sous-extensions convenables
de corps cyclotomiques. Soit $K$ un corps quadratique imaginaire et $E=FK$.
Alors $E$ est un corps CM qui est Galoisien sur $\QQ$ de groupe $\ZZ/2\ZZ\times \ZZ/p\ZZ$.
Dans cette situation, pour tout type CM sur $E$, on a la suite exacte scind\'ee
$$
1\rightarrow \{1,\rho\}\rightarrow \Fg\rightarrow \ZZ/p\ZZ \rightarrow 1
$$
et par le lemme \ref{l4.4} le cocycle $js:\ZZ/p\ZZ\rightarrow (\ZZ/2\ZZ)^p/\ZZ/2\ZZ$
est trivial. Il existe donc $a\in (\ZZ/2\ZZ)^p$ tel que $s(\sigma)=a\sigma(a)$ 
pour tout $\sigma\in \ZZ/p\ZZ$. Apr\`es renum\'erotation de $\{1,\dots,p\}$,
on peut supposer que $\ZZ/p\ZZ$ est engendr\'e par le $p$-cycle $(1,2,\dots,p)$.
On peut par ailleurs
par un choix convenable du type CM sur $E$ utilisant  la remarque \ref{r4.2} supposer que
$a=(1,-1)$. Un calcul simple montre alors que $L_{\mu}$ est engendr\'e par 
$\mu$, $\sum_{k=1}^p e_{k}$ et les $(e_{1}+e_{i})$ avec $i\in \{2,\dots,p\}$. 
On v\'erifie alors que $L_{\mu}$ est de rang maximal $p+1$ et que $X^*(GU_{E})/L_{\mu}\simeq 
\ZZ/(p-2)\ZZ$.

\section{Points sp\'eciaux des vari\'et\'es de Shimura.}

Dans cette section on commence \`a aborder le probl\`eme de minoration des
orbites Galoisiennes de points sp\'eciaux dans les vari\'et\'es de Shimura.
On consid\`ere une don\'ee de Shimura $(G,X)$. On peut sans perte de g\'en\'eralit\'e 
supposer que $G$ est le groupe de Mumford-Tate g\'en\'erique de $X$.
On peut aussi supposer que $K$ est net. Avec ces hypoth\`eses $\Gamma := K \cap G(\QQ)$
agit sans points fixes sur $X$.
On peut aussi sans perte de g\'en\'eralit\'e ne s'int\'eresser qu'\`a des points de la composante $S$
de $\Sh_{K}(G,X)$ qui est l'image de $X^+\times \{ 1 \}$ dans $\Sh_K(G,X)$ (o\`u $X^+$ d\'esigne une composante connexe de $X$).
On fixe dans la suite une repr\'esentation fid\`ele $G\hookrightarrow \GL_n$.
Ceci permet de definir les mod\`eles entiers de $G$ et de ses sous-groupes alg\'ebriques.
On suppose \'egalement que $K$ est le produit $K=\prod_p K_p$ o\`u $K_p$ est un
sous-groupe compact ouvert de $G(\QQ_p)$.

Soit $(T,\{h\})\subset (G,X)$ une donn\'ee de Shimura sp\'eciale telle que
$T$ est le groupe de Mumford-Tate de $h$. On note $K_{T}=K\cap T(\AAA_{f})$.
La non maximalit\'e du sous-groupe compact ouvert $K_{T}$
contribue \`a la taile de l'orbite sous Galois du point sp\'ecial  $x=\overline{(h,1)}$.
Nous d\'ecrivons le r\'esultat pr\'ecis dans une premi\`ere partie puis
nous rappelons  dans une deuxi\`eme  des r\'esultats de 
Clozel et du premier auteur \cite{CU} concernant l'image des morphismes de r\'eciprocit\'e
au niveau des groupes de classes sous des hypoth\`eses de connexit\'e du noyau du
morphisme de r\'eciprocit\'e.  

\subsection{Passage de $K_{T}^{m}$ \`a $K_{T}$}\label{s5.1}
\begin{prop}\label{p5.1}
Soit $(G,X)$ une donn\'ee de Shimura. Soit $K$ un sous-groupe compact ouvert de 
$G(\AAA_f)$.

Soit $(T,\{h\})\subset (G,X)$ une donn\'ee de Shimura sp\'eciale et $L$ le corps de
d\'ecomposition de $T$. On suppose que $T$ est le groupe de Mumford-Tate de $h$.
Soit $K^m_T$ le sous-groupe compact ouvert maximal de $T(\AAA_f)$ et soit
$K_T = K \cap T(\AAA_f)$.
Soit $r \colon R_{L}:=\Res_{L/\QQ} \GG_{m,L} \lto T$ le morphisme de r\'eciprocit\'e et $U$ l'image de $r((\AAA_f\otimes L)^*)$
dans $T(\QQ)\backslash T(\AAA_f) / K^m_T$.

On a 
$$
|\Gal(\ol\QQ / L) \cdot \overline{(h,1)}| \gg B^{i(T)}|K_T^m/K_T| \cdot |U|
$$ 
o\`u $B$ est une constante uniforme et $i(T)$ d\'esigne le nombre de premiers $p$ tels que 
$K_{T,p}^m \neq K_{T,p}$.
\end{prop}
\begin{lem}
Soit $N$ le noyau du morphisme naturel
$$
T(\QQ)\backslash T(\AAA_f) /K_T \lto T(\QQ) \backslash T(\AAA_f) /K^m_T
$$
Alors
$$
|N| = c |K^m_T/K_T|
$$
o\`u $c=c(x)$ est une constante uniformement born\'ee quand $x=\overline{(h,1)}$ varie parmi
les points CM de $S$. 
\end{lem}
\begin{proof}
Il est facile de voir que
$$
N = (T(\QQ)\cap K^m_T )\backslash K_T^m /K_T
$$
Le groupe $T(\QQ)\cap K^m_T$ est fini d'ordre born\'e uniform\'ement
quand $x=\overline{(h,1)}$ varie parmi les points CM de $S$.
En effet il existe un sous-groupe compact ouvert net $K_{T}^{net}$ de $T(\AAA_{f})$
d'indice uniform\'ement born\'e dans $K_{T}^m$. Le groupe $K_{T}^{net}\cap T(\QQ)$
est trivial car de torsion dans un sous-groupe  compact ouvert net. On en d\'eduit
que $\vert T(\QQ)\cap K_{T}^{m}\vert \leq \vert K_{T}^{m}/K_{T}^{net}\vert $ donc que 
$T(\QQ)\cap K^m_T$ est fini d'ordre born\'e uniform\'ement.
\end{proof}
\begin{proof}
Sans perte de g\'en\'eralit\'e, on peut supposer que le groupe compact ouvert $K$ est net.

Notons d'abord qu'il suffit de montrer que 
$|\Gal(\ol\QQ / L) \cdot \overline{(h,1)}|$ est au moins de la taille de l'image de $r((\AAA_f\otimes L)^*)$ dans
$T(\QQ)\backslash T(\AAA_f) /K_T$.
En effet, supposons que ce soit le cas. 
On constate alors que 
$$
|\Gal(\ol\QQ / L) \cdot \overline{(h,1)}|\ge |U||\Theta|
$$
o\`u $\Theta$ est l'image de $r((\AAA_f\otimes L)^*)\cap K_T^m$ dans $K_T^m/K_T$.
D'apr\`es \cite{UY}, lemme 2.18 , on a 
$$
|\Theta| \gg B^{i(T)}|K_T^m / K_T|
$$
o\`u $B$ est une constante uniforme et $i(T)$ est comme dans l'\'enonc\'e.

D\'emontrons maintenant que $|\Gal(\ol\QQ / L) \cdot x|$ est au moins de la taille de l'image de $r((\AAA_f\otimes L)^*)$ dans
$T(\QQ)\backslash T(\AAA_f) /K_T$.
L'inclusion de donn\'ees de Shimura $(T,\{h\})\subset (G,X)$ induit un morphisme de vari\'et\'es de
Shimura:
$$
\Sh_{K_T}(T,\{h\})\lto \Sh_K(G,X)
$$
Ce morphisme est d\'efini sur le compos\'e du corps reflexe de $(T,\{h\})$ et celui de $(G,X)$.
De plus, par le lemme 2.2 de \cite{UY}, ce morphisme est 
injectif. 
On en d\'eduit que la taille de l'orbite sous Galois de $x=\ol{(h,1)}$ est, a une constante uniforme pr\`es,
la taille de l'image de  $r((\AAA_f\otimes L)^*)$ dans
$T(\QQ)\backslash T(\AAA_f) /K_T$.

\subsection{Morphisme de r\'eciprocit\'e \`a noyaux connexes}\label{s5.2}

Si $T$ est un tore sur $\QQ$, on note $\pi(T)$ le groupe
$T(\mathbb{A}_{f})/T(\mathbb{Q})^{-} $ (adh\'erence topologique), modifiant un peu la notation
de Deligne. Si $(T,\{h\})$ est une sous-donn\'ee de Shimura de $(G,X)$ telle que 
$T$ est le groupe de Mumford-Tate de $h$, et $r:R_{L}\rightarrow T$
le morphisme de r\'eciprocit\'e on note
$r_{\AAA_{f}/\QQ}: \pi(R_{L})\rightarrow \pi(T)$ le morphisme induit.

 On rappelle que l'on note
$h_{T}=T(\QQ)\backslash T(\AAA_{f})/K_{T}^{m}$ le groupe de classes
de $T$. On dispose alors du r\'esultat suivant (\cite{CU} thm. 3.3).

\begin{teo}
Si $(T,\{h\})$ varie parmi les sous-donn\'ees \textrm{CM} de $(G,X)$ telles que le noyau
\begin{equation}
N=\ker(r:R_{L}\rightarrow T)
\end{equation}
est connexe, le conoyau de ${r_{\AAA_{f}/\QQ}}:\pi(R)\rightarrow \pi(T)$
est de taille uniform\'ement born\'ee.
\end{teo}

On en d\'eduit en particulier que le conoyau de $\overline{r}:h_{R_{L}}\rightarrow h_{T}$
est uniform\'ement born\'e quand $(T,\{h\})$ varie parmi les sous-donn\'ees CM
telles que le noyau de $r$ est connexe. Dans formulation de (\cite{CU} thm 3.3)
le corps reflex de $(T,\{h\})$ \`a la place de $L$. La preuve donn\'ee
dans ce texte vaut pour $L$ \`a la place du corps reflex. Il est simple
de montrer que les \'enonc\'es  du th\'eor\`eme  pour $L$ et pour le corps reflex
sont en fait \'equivalents.

En combinant ce r\'esultat avec la proposition \ref{p5.1}
et le th\'eor\`eme \ref{teo2.3} on obtient un des r\'esultats principaux que nous avons
en vue dans ce texte.

\begin{cor}
Soit $(G,X)$ une donn\'ee de Shimura telle que $G$ est le groupe de 
Mumford-Tate g\'en\'erique sur $X$. Soit $K$ un sous-groupe compact ouvert de 
$G(\AAA_f)$. Soit $d$ le rang absolu de $G$. 

Soit $(T,\{h\})\subset (G,X)$ une sous--donn\'ee de Shimura sp\'eciale et $L$ le corps de
d\'ecomposition de $T$. On suppose que $T$ est le groupe de Mumford-Tate de $h$
et que le noyau du morphisme de r\'eciprocit\'e $r:R_{L}\rightarrow T$
est connexe. Alors pour tout $\epsilon >0$
\begin{equation}\label{eqGal}
|\Gal(\ol\QQ / L) \cdot \overline{(h,1)}| \gg B^{i(T)}|K_T^m/K_T| \cdot |h_{T}|\ge
c(d,\epsilon) C^{i(T)}|K_T^m/K_T| D_{L}^{\frac{\lambda(d)}{2}-\epsilon}.
\end{equation}
o\`u $C$ est une constante ne d\'ependant que de $d$
 et $i(T)$ d\'esigne le nombre de premiers $p$ tels que 
$K_{T,p}^m \neq K_{T,p}$. La constante positive $\lambda(d)$ est 
explicit\'ee dans la d\'efinition \ref{d3.4} et $c(\epsilon,d)$ est une constante
strictement positive ne d\'ependant que de $d$ et de $\epsilon$.
\end{cor}

En utilisant la proposition \ref{p4.7} et le fait que $\lambda(2)=\lambda(3)=\frac{2}{5}$,
on trouve le r\'esultat suivant qui g\'en\'eralise le cas bien connu $g=1$.

\begin{cor}
Soit $\epsilon>0$.
Soit  $x=\overline{(h,1)}$ un point CM du module $\AAA_{g}$ des vari\'et\'es ab\'eliennes
principalement polaris\'ees de dimension $g=2$ ou $g=3$ correspondant
\`a une vari\'et\'e ab\'elienne simple. Soit $T=T_{x}$ le groupe de Mumford-Tate 
de $h$. Alors 
\begin{equation}
|\Gal(\ol\QQ / L) \cdot \overline{(h,1)}\vert \ge
c(\epsilon) C^{i(T)}|K_T^m/K_T| D_{L}^{\frac{1}{5}-\epsilon}.
\end{equation}
pour une constante $c(\epsilon)$ ne d\'ependant pas de $x$.
\end{cor}
 
 Pour $g=1$, on a $\lambda(1)=1$ et on retrouve les estimations classiques
 (en $D_{L}^{\frac{1}{2}-\epsilon}$)
 de la taille du groupe de Picard d'un corps quadratique imaginaire.
 Pour obtenir des minorations de l'orbite sous Galois d'un point $CM$
de $\AAA_{g}$ on peut en g\'en\'eral se ramener au cas des vari\'et\'es
ab\'eliennes simples. 
 
Soit $(G,X)$ une donn\'ee de Shimura et $(T,\{h\})\subset (G,X)$ une
 sous--donn\'ee de Shimura sp\'eciale. On dit que $T$ est Galois
 g\'en\'erique si l'image $I$ de $\Gal(\oQ/\QQ)$ dans $\Aut(X^{*}(T))$ est
 maximale (parmi les images possibles pour un sous-tore de $G$). Nous 
 faisons r\'ef\'erence \`a  (\cite{CU} 2.1)  pour une d\'efinition plus pr\'ecise.
 Notons que d'apr\`es la proposition 2.1 de \cite{CU} les 
 sous-donn\'ees sp\'eciales Galois g\'en\'erique de $(G,X)$ existent pour tout $(G,X)$,
 elles sont m\^{e}me g\'en\'eriques en un sens expliqu\'e dans \cite{CU}.
  Il est montr\'e dans \cite{CU} dans de nombreux
 cas que le morphisme de r\'eciprocit\'e est \`a noyau connexe pour
 les sous-donn\'ees de Shimura $(T,\{h\})\subset (G,X)$ avec
 $T$ Galois g\'en\'erique. C'est par exemple le cas si $G$ est $\QQ$-simple adjoint
 de type $B_{l}$, $C_{l}$ et dans certains cas de type $D_{l}$ et $A_{l}$ que le lecteur
 pourra consulter dans
 \cite{CU}.  Dans tous ces cas si on fait varier le point sp\'ecial $x$
 parmi des sous-donn\'ees Galois g\'en\'eriques on obtient des 
 minorations inconditionnelles pour la taille de l' orbite sous Galois de $x$
 de la forme donn\'ee dans l'\'equation (\ref{eqGal}).
 Nous n'\'ecrirons pas l'\'enonc\'e le plus g\'en\'eral possible. Retenons seulement
 le r\'esultat suivant qui concerne $\AAA_{g}$ qui est une cons\'equence
 des r\'esultats pr\'ec\'edents et de la proposition \ref{p4.3}.
 
\begin{cor}
Soit $g$ un entier.
Soit $A$ une vari\'et\'e ab\'elienne principalement polaris\'ee de dimention $g$. On suppose 
que $\End(A)\otimes \QQ=E$ est un corps CM de degr\'e $2g$ v\'erifiant les hypoth\`eses
de la proposition \ref{p4.3}.
Soit $x$ le point sp\'ecial de $\AAA_{g}$ associ\'e \`a $A$. On dira que
$x$ est ``suffisamment Galois g\'en\'erique''.
Noter que si 
$x$ est Galois g\'en\'erique $A$ a la propri\'et\'e requise.  En conservant
les notations des \'enonc\'es pr\'ec\'edents, quand $x$ varie parmi
les points sp\'eciaux suffisamment Galois g\'en\'erique

\begin{equation}\label{eqGal2}
|\Gal(\ol\QQ / L) \cdot \overline{(h,1)}| \ge
c(d,\epsilon) C^{i(T)}|K_T^m/K_T| D_{L}^{\frac{\lambda(d)}{2}-\epsilon}.
\end{equation}

\end{cor}

\end{proof}

\section{Bornes pour le noyau de r\'eciprocit\'e sous GRH.}\label{s6}

Dans cette section on am\'eliore les bornes pour les orbites de Galois
de points sp\'eciaux donn\'ees dans \cite{Yaf}.

\begin{teo} \label{borne_torsion}
Admettons l'hypoth\`ese de Riemann g\'en\'eralis\'ee pour les corps CM.
Soit $(G,X)$ une donn\'ee de Shimura et $x={(h,1)}$ un point CM de $Sh_K(G,X)$.
Soit $T$ le groupe de Mumford-Tate de $h$.
 
Soit $L$ le corps de d\'ecomposition de $T$ et  $D_L$ 
la valeur absolue du discriminant de $L$.

On a
$$
\Gal(\ol \QQ/F)\cdot \ol{(x,1)} \gg B^{i(T)}|K_T^m/K_T| D_L^{\mu}
$$
o\`u $\mu>0$ est uniforme.
\end{teo}
%\begin{rem}
%Ce th\'eor\`eme montre en particulier que pour tout entier $n$, il existe
%un $\epsilon=\epsilon(n) > 0$, tel que 
%$$
%|G[n]| \leq (1 - \epsilon)|G|
%$$
%\end{rem}
\begin{proof}
L'apparition du facteur $B^{i(T)}|K_T^m/K_T|$ a \'et\'e trait\'ee pr\'ecedemment (c.f. \ref{p5.1}).

Soit, comme avant, $r \colon R_L \lto T$ le morphisme de r\'eciprocit\'e et $\ol{r}\colon h_L \lto h_T$ le morphisme induit par $r$ au niveau des groupes de
classes de $R_L$ et de $T$ respectivement.
Il suffit de montrer que
$$
|Im(\ol{r})| \gg D_L^{\mu}
$$

La preuve s'inspire de \cite{AD} et de \cite{Yaf}.

On d\'efinit, suivant \cite{AD}, pour un groupe ab\'elien $H$ et un entier $l$,
$M_H(l)$ comme \'etant le plus petit entier $A$ tel que
pour tout $l$-uplet $(g_1,\dots , g_l)$ d'\'el\'ements de $H$, il existe $(a_1,\dots, a_l)\in \ZZ^n\backslash \{0\}$
avec $\sum_j |a_j|\leq A$ v\'erifiant $g_1^{a_1}\cdots g_l^{a_l}=1$.

Prenons $l=|H|$ et soit $g_1,\dots , g_l \in H$.
Si $g_i=1$ pour un certain $i$, alors on a  une relation
multiplicative non-triviale en les $g_{i}$ avec $A=1$.
Autrement, on a une relation de la forme $g_i g_j^{-1} = 1$ pour
des indices $i\neq j$.
Dans tous les cas, on a une relation $g_1^{a_1}\cdots g_l^{a_l} = 1$ avec
$\sum |a_i| \leq 2$.
On voit donc que pour $l = |H|$, on a 
$$
M_H(l) \leq 2.
$$

Prenons maintenant $H=h_L/\ker(\overline{r})$.

On va d\'emontrer l'estimation suivante:
\begin{equation}\label{Eq}
M_H(l) > c \frac{\log(D_L)}{\log(l)+\log\log(D_L)}
\end{equation}
o\`u $c>0$ est une constante uniforme.

Cette estimation implique celle d\'esir\'ee pour 
$|H|$:

$$
|H| > \frac{D_L^{c/2}}{\log(D_L)} \gg D_L^{\mu}
$$
avec $\mu>0$ uniforme.

On va maintenant d\'emontrer l'in\'egalit\'e (\ref{Eq}).
Rappelons quelques notions et r\'esultats de la section 2 de \cite{Yaf}. 
D'apr\`es la proposition 2.2 de \cite{EY} on peut supposer $G$ adjoint.
Le morphisme de r\'eciprocit\'e $r\colon R_L \lto T$ induit 
une inclusion $X^*(T)\subset X^*(R_L)$ et on a une base canonique 
de $X^*(R_L)$ donn\'ee par \'enum\'eration des \'el\'ements de $\Gal(L/\QQ)$.
Il existe une base   $\cB$ de $X^*(T)$ 
 telle que les coordonn\'ees des caract\`eres $\chi$ de $\cB$
par rapport \`a la base canonique de $X^*(R_L)$, sont born\'ees uniform\'ement.
De plus comme $G$ est suppos\'e adjoint
pour tout caract\`ere $\chi$ de $\cB$, $\chi\ol{\chi}$ est le caract\`ere trivial.

Soit $l\geq 1$ un entier et $p_1,\dots , p_l$, $l$ premiers qui d\'ecomposent $T$ et 
$a_1,\dots , a_l$ des entiers relatifs.
Pour chaque $i$, on fixe une place $v_i$ de $L$ au d\'essus de $p_i$ et un id\`ele $P_i$ 
dans $(L\otimes\AAA_f)^*$ qui est l'uniformisante \`a la place $v_i$ et $1$ ailleurs.
 Considerons $I = P_1^{a_1}\cdots P_r^{a_l} \subset (L\otimes\AAA_f)^*$ et sa classe $\ol{I}$
 dans $h_L$.
Supposons que $\overline{I}$ soit dans le noyau de $\overline{r}$ i.e 
$$
r(I) = \pi k
$$
o\`u $\pi\in T(\QQ)$ et $k\in K^m_T$.
Soit $\pi_i = \chi_i(\pi)\subset L^*$.
Le lemme 2.15 de \cite{Yaf} montre que $\QQ[\pi_1,\dots , \pi_r]=L$.

Soit $t$ une borne uniforme sur les coordonn\'ees des $\chi_i$.
On voit alors que $\pi'_i:=(p_1^{|a_1|}\cdots p_l^{|a_l|})^t \pi_i \in O_L$
et le fait que $\chi_i \ol{\chi_i}$ est le caract\`ere trivial implique que
$$
|\sigma(\pi'_i)| \leq (p_1^{|a_1|}\cdots p_l^{|a_l|})^{2t}
$$
pour tout $\sigma \in \Gal(L/\QQ)$.

Soit $n_{L}$ le degr\'e de l'extension $L$ sur $\QQ$. Comme $L$ est le corps de d\'ecomposition
d'un tore de dimension $d$ fix\'e, $n_{L}$ est uniform\'ement born\'e.
On peut choisir une base $b_1,\cdots, b_{n_L}$ de $L$ sur $\QQ$
avec  $b_{k}=\prod_{i=1}^{d} {\pi'_i}^{n_{i,k}}$ pour des entiers naturels $n_{i,k}$
tels que $n_{i,k}\leq n_{L}$. 
Il suffit de remarquer en effet 
que pour tout $i$, les \'el\'ements $1, \pi'_i, \dots , {\pi'_i}^{n_L}$ sont lin\'eairement d\'ependants.

Le fait que les $b_i$ sont dans $O_L$ et 
que $b_1,\cdots, b_{n_L}$ forment une base de $L$ sur $\QQ$, implique que
 $\ZZ[b_1,\dots , b_{n_L}]$ est un ordre dans $O_L$.
En particulier
\begin{equation}\label{Eq2}
|\Discr (\ZZ[b_1,\dots , b_{n_L}])| \geq D_L
\end{equation}
D'autre part, $|\Discr (\ZZ[b_1,\dots , b_{n_L}])|$ est le d\'eterminant de la matrice
$\Big{(}Tr_{L/\QQ}(b_ib_j)\Big{)}$.
Par l'in\'egalit\'e d'Hadamard, si $b$ est un majorant de tous les 
$|Tr_{L/\QQ}(b_ib_j)|$, alors
$$
|\Discr (\ZZ[b_1,\dots , b_{n_L}])| \leq c(n_L) b^{n_L}
$$
o\`u $c(n_L)$ ne d\'epend que de $n_L$ (on peut prendre $c(n_L) = n_L^{n_L}$).

Du fait que $|\sigma(\pi'_i)| \leq (p_1^{|a_1|}\cdots p_l^{|a_l|})^{2A}$, on d\'eduit qu'il existe un entier uniforme $D$ (ne d\'ependant que de $n_L$) tel que  
$$
|Tr_{L/\QQ}(b_ib_j)|\leq (p_1^{|a_1|}\cdots p_l^{|a_l|})^D
$$

et donc, par l'in\'egalit\'e d'Hadamard, apr\`es avoir remplace $D$ par $D n_L$
$$
|\Discr \ZZ[b_1,\dots , b_{n_L}]| \leq  c(n_L)(p_1^{|a_1|}\cdots p_l^{|a_l|})^{D}
$$

Notons $c = 1/c(n_L)$.
L'\'equation (\ref{Eq2}) donne alors:
$$
(p_1^{|a_1|}\cdots p_l^{|a_l|})^D \geq c D_L
$$

Nous allons maintenant choisir $l$ et $p_i$.

Rappelons une cons\'equence  du th\'eor\`eme de Chebotarev effectif. Le lecteur pourra consulter
le Lemme 2.1 de \cite{AD} pour une preuve.
\begin{teo} Admettons  l'hypoth\`ese de Riemann g\'en\'eralis\'ee.
On note $\pi_{L}(x)$ le nombre de premiers $p$ totalement decompos\'es dans $L$ tels que $p\leq x$.
Il existe des constantes absolues (et effectivement calculables) $c_1>0$ et $c_2>0$  tels que 
 $$
 \pi_{L}(x) \geq c_2 \frac{x}{\log(x)}
 $$
 pour tout 
 $x\geq c_1 \log(D_L)^2 (\log\log D_L)^4$.
\end{teo}

Soit maintenant $l\geq 1$ un premier.
$$
x = c_3 l \log(l) + c_1 \log(D_L)^2 (\log\log D_L)^4
$$
o\`u $c_3$ est une constante uniforme que nous allons 
expliciter.
Nous voulons choisir la constante $c_3$ telle que $c_2 \frac{x}{\log(x)} \geq l$.
Notons que (si $x\ge e$)
$$
\frac{x}{\log(x)}\geq \frac{c_3l\log(l)}{\log(c_3) + 2\log(l)}.
$$
On a alors $c_2 \frac{x}{\log(x)} \geq l$ d\`es que $\frac{c_2c_3}{\log(c_3) + 2}\geq 1$.
On peut alors par exemple prendre $c_3 = \max{(e,\frac{4}{{c_2}^{2}})}$.

On peut trouver $p_1,\dots , p_l$ d\'ecompos\'es dans $L$ et v\'erifiant
$$
p_i \leq x.
$$

Soit maintenant $I$ un \'el\'ement  comme avant.
On note $A=\sum_{i=1}^{l} \vert a_{i}\vert$. Ce qui pr\'ec\`ede donne
$$
AD \log(x) \geq \log(cD_L)
$$
Par ailleurs, il y a une constante uniforme $c_5$ telle que
$$
\log(x) \leq c_5(\log(l) + \log\log(D_L))
$$
On obtient donc une borne inf\'erieure pour $A$ de la forme souhait\'ee.
Ceci ach\`eve la preuve de l'in\'egalit\'e (\ref{Eq}) et du th\'eor\`eme.

%
%\begin{lem}
%Considerons $k$ premiers d\'ecompos\'es dans $L$:
%$p_1,\dots , p_k$ et les \'el\'ements $P_1,\dots , P_k$ de $T(\AAA_f)$ comme si d\'essus.
%Supposons que $P_1^{a_1}\cdots P_k^{a_k} = \pi k$ o\`u $\pi\in T(\QQ)$ et $k\in K_T^m$.
%Pour tout $i$, soit $\alpha_i = \chi_i(\pi) /\ol{\chi_i(\alpha)} \in L$
%Alors
%$$
%d h(\alpha_i) = \sum_{j=1}^k |a_j| \log(p_i)
%$$
%o\`u $h$ est la hauteur additive usuelle.
%\end{lem}
%\begin{proof}
%
%\end{proof}

%On a $T(\QQ_p) = X_*(T)\otimes \QQ_p^*$.

\end{proof}

\end{document}